\documentclass[11pt]{amsart}
\input epsf
\usepackage{latexsym}
\usepackage{amssymb,amsmath,amsthm,amscd, amssymb,
graphicx, enumerate, verbatim}
\usepackage[all]{xy}

\numberwithin{equation}{section}
\newtheorem{cor}[equation]{Corollary}
\newtheorem{lem}[equation]{Lemma}
\newtheorem{prop}[equation]{Proposition}
\newtheorem{convention}[equation]{Convention}
\newtheorem{thm}[equation]{Theorem}
\newtheorem{prob}[equation]{Problem}
\newtheorem{quest}[equation]{Question}

\newtheorem{Example}[equation]{Example}
\newenvironment{ex}{\begin{Example}\rm}{\end{Example}}
\newtheorem{remark}[equation]{Remark}
\newenvironment{rmk}{\begin{remark}\rm}{\end{remark}}
\def\co{\colon\thinspace}

\newcommand{\norm}[1]{\lVert#1\rVert}

\newcommand{\Ric}{\mbox{Ric}}

\newcommand{\e}{\epsilon}

\newcommand{\tinycirc}{{_{^\circ}}}

\def\a{\alpha}

\def\g{\gamma}

\def\z{\zeta}

\def\b{\beta}

\def\d{\partial}

\def\r{\rho}

\def\s{\sigma}

\def\S1{\bf S^1}

\newcommand{\R}{{\mathbb R}}

\begin{document}

\abovedisplayskip=6pt plus3pt minus3pt
\belowdisplayskip=6pt plus3pt minus3pt

\title[Moduli spaces and non-unique souls]
{\bf Moduli spaces of nonnegative sectional curvature
and non-unique souls}

\rm

\author{Igor Belegradek\and Slawomir Kwasik \and Reinhard Schultz}

\address{Igor Belegradek\\School of Mathematics\\ Georgia Institute of
Technology\\ Atlanta, GA 30332-0160}\email{ib@math.gatech.edu}
\address{Slawomir Kwasik\\Mathematics Department\\
Tulane University\\6823 St. Charles Ave\\New Orleans, LA 70118}
\email{kwasik@tulane.edu}
\address{Reinhard Schultz\\Department of Mathematics\\ University of California Riverside\\ 900 Big Springs Drive, CA 92521}
\email{schultz@math.ucr.edu}
\date{}
\begin{abstract} 
We apply various topological methods to distinguish 
connected components of moduli spaces of complete Riemannian 
metrics of nonnegative sectional curvature on open manifolds.
The new geometric ingredient is that souls of nearby nonnegatively 
curved metrics are ambiently isotopic.
\end{abstract}
\maketitle

\section{Introduction}
\label{sec: intro}

A fundamental structure result, due to Cheeger-Gromoll~\cite{CheGro}, 
is that any open complete manifold of $\sec\ge 0$
is diffeomorphic to the total space of a normal bundle to
a compact totally geodesic submanifold, called a soul.
A soul is not unique e.g. in the Riemannian
product $M\times\mathbb R^k$ of a closed manifold $M$ with $\sec\ge 0$
and the standard $\mathbb R^k$, the souls are of the form $M\times \{x\}$. 
Yet Sharafutdinov~\cite{Sha-souls}
proved that any two souls can be moved to each other by a diffeomorphism
that induces an isometry of the souls.

The diffeomorphism class of the soul may depend on the metric 
e.g. any two homotopy equivalent $3$-dimensional lens spaces 
$L, L^\prime$ become diffeomorphic after multiplying by 
$\mathbb R^3$~\cite{Mil-haup}, so taking non-homeomorphic 
$L, L^\prime$ gives two product metrics on 
$L\times\mathbb R^3=L^\prime\times\mathbb R^3$ with 
non-homeomorphic souls. It turns out that codimension $3$
is optimal, indeed, Kwasik-Schultz proved in~\cite{KwaSch-toral} 
that if $S$, $S^\prime$ are
linear spherical space forms such that $S\times\mathbb R^2$, 
$S^\prime\times\mathbb R^2$ are diffeomorphic, then $S$, $S^\prime$ are diffeomorphic. Another well-known example is that all homotopy $7$-spheres
become diffeomorphic after taking product with $\mathbb R^3$ 
(see Remark~\ref{rmk: exotic sphere times R3});
since some homotopy $7$-spheres~\cite{GroZil} have metrics of
$\sec\ge 0$, so do their products with $\mathbb R^3$, which 
therefore have nonnegatively curved metrics with
non-diffeomorphic souls. Codimension $3$
is again optimal, because any 
simply-connected manifold $S$ of dimension $\ge 5$ can be recovered
(up to diffeomorphism) from $S\times \mathbb R^2$
(see~\cite{KwaSch-toral} or Remark~\ref{rmk: ko=wh=0}).

Belegradek in~\cite{Bel}
used examples of Grove-Ziller~\cite{GroZil} to produce first 
examples of infinitely many nondiffeomorphic souls for metrics on 
the same manifold, e.g. on $S^3\times S^4\times\mathbb R^5$.
Other examples of simply-connected manifolds with infinitely many
nondiffeomorphic souls, and better control on geometry,
were constructed by Kapovitch-Petrunin-Tuschmann in~\cite{KPT}.

One motivation for the present work was to construct non-diffeomorphic
souls of the smallest possible codimension; of course, 
multiplying by a Euclidean space
then yields examples in any higher codimension.
We sharpen examples 
in~\cite{Bel} by arranging the soul to have codimension $4$ 
and any given dimension $\ge 7$.

\begin{thm}\label{intro-thm: codim 4 inf many souls} 
For each $k\geq 3$, there are infinitely many complete
metrics of $\sec\ge 0$ on 
$N=S^4\times S^k\times \mathbb R^4$ whose souls are
pairwise non-homeomorphic. 
\end{thm}

Similarly, in Theorem~\ref{thm: kpt modified} 
we sharpen Theorems B and C in~\cite{KPT}
to make the souls there of codimension $4$, 
in particular, we prove:
 
\begin{thm}\label{intro-thm: kpt modified} 
There exists an open simply-connected manifold $N$ that admits 
infinitely many complete metrics of $\sec\in [0,1]$
with pairwise non-homeomorphic codimension $4$
souls of diameter $1$. Moreover, one can choose $N$ so that 
each soul has nontrivial normal Euler class. 
\end{thm}

We do not know examples of manifolds with 
infinitely many non-diffeomorphic souls of codimension $<4$, 
and in an effort to find such examples we systematically 
study vector bundles with diffeomorphic total spaces,
and among other things prove the following:

\begin{thm}\label{intro-thm: codim <4 finiteness}
Suppose there is a manifold $N$ that admits 
complete nonnegatively curved metrics with
souls $S_k$ of codimension $<4$
such that the pairs $(N, S_k)$ lie in infinitely many 
diffeomorphism types.
If $\pi_1(N)$ is finite, $S_k$ is orientable,
and $\dim(S_k)\ge 5$, then\newline
$\mbox{\quad}\textup{(1)}$ $\pi_1(N)$ is nontrivial
and $\dim(S_k)$ is odd;\newline
$\mbox{\quad}\textup{(2)}$
the products $S_k\times\mathbb R^3$ lie in finitely
many diffeomorphism types. 
\end{thm}

In Example~\ref{ex: finiteness fails} we describe two 
infinite families of closed manifolds with the property that
if each manifold in the family admits a metric of $\sec\ge 0$, 
then they can be realized as codimension $1$ souls in the 
same open manifold $N$. In general, if $M$
is a closed oriented smooth manifold of dimension $4r-1\ge 7$
whose fundamental group contains a nontrivial finite order element, 
then there are infinitely many pairwise non-homeomorphic closed 
manifolds $M_i$ such that $M_i\times\mathbb R^3$ is diffeomorphic to 
$M\times\mathbb R^3$ (see~\cite{ChaWei}); thus if each $M_i$ 
admits a metric of $\sec\ge 0$, then $M\times\mathbb R^3$
carries infinitely many (product) metrics with nondiffeomorphic
souls.

Another goal of this paper is to study moduli spaces 
of complete metrics of nonnegative sectional curvature 
on open manifolds.
Studying moduli spaces of Riemannian metrics that satisfy 
various geometric assumptions is largely a topological activity, 
see e.g.~\cite{WanZil}, \cite{KreSto},
\cite{NabWei-ihes, NabWei-fractal},
\cite{FarOnt-mod, FarOnt-Teich}, 
\cite{Ros} and references therein.

Let ${\mathfrak R}^{k,u}(N)$ 
denote the space of complete Riemannian 
$C^\infty$ metrics on a smooth manifold $N$ with
topology of uniform $C^k$-convergence, where $0\le k\le\infty$, 
and let ${\mathfrak R}^{k,c}(N)$ denote the 
same set of metrics
with topology of $C^k$-convergence on compact subsets. 
Let $\mathfrak R_{\sec\ge 0}^{k,u} (N)$, 
$\mathfrak R_{\sec\ge 0}^{k,c} (N)$ be the subspaces of 
$\mathfrak R^{k,u} (N)$, $\mathfrak R^{k,c} (N)$ respectively,
consisting of metrics of $\sec\ge 0$, and let
$\mathfrak M_{\sec\ge 0}^{k,u} (N)$, 
$\mathfrak M_{\sec\ge 0}^{k,c} (N)$,
$\mathfrak M^{k,u} (N)$, $\mathfrak M^{k,c} (N)$  
denote the corresponding moduli spaces, 
i.e. their quotient spaces by the $\mathrm{Diff}(N)$-action via pullback.
We adopt:

\begin{convention} 
\label{intro-convention}
If an assertion about a moduli space or 
a space of metrics holds for any $k$, then the superscript
$k$ is omitted from the notation, and if $N$ is compact, then
$c$, $u$ are omitted.\end{convention}

The space ${\mathfrak R}^{c}(N)$ is closed under convex combinations~\cite{FegMil}
and hence is contractible, in particular, ${\mathfrak M}^{c}(N)$ is 
path-connected. By contrast, if $N$ is non-compact,
${\mathfrak M}^{u}(N)$ typically has uncountably many connected
components because metrics in the same component of 
${\mathfrak M}^{u}(N)$ lie within a finite uniform distance 
of each other, while the uniform distance is infinite
between metrics with different asymptotic geometry 
(such as rotationally symmetric metrics on $\mathbb R^2$ with non-asymptotic 
warping functions).

It was shown in~\cite{KPT} that metrics with non-diffeomorphic
souls lie in different components of 
$\mathfrak M_{\sec\ge 0}^{c} (N)$ provided any
two metrics of $\sec\ge 0$ on $N$ have souls that intersect,
which can be forced by purely topological assumptions on $N$ 
e.g. this holds if $N$ has a soul with nontrivial normal Euler class,
of if $N$ has a codimension $1$ soul. 

A simple modification of
the proof in~\cite{KPT} shows (with no extra assumptions on $N$) that
metrics with non-diffeomorphic souls lie in different components of 
$\mathfrak M_{\sec\ge 0}^{u} (N)$.
In fact, this result and the result of~\cite{Sha-souls} 
that any two souls of the same metric
can be moved to each other by a diffeomorphism of the ambient 
nonnegatively curved manifold have the following common generalization.

\begin{thm}
\label{intro-thm: ambient isotopy} 
\textup{(i)} If two metrics are sufficiently close in
$\mathfrak R^u_{\sec\ge 0} (N)$, their souls are ambiently 
isotopic in $N$.\newline 
\textup{(ii)}
The map associating
to a metric $g\in \mathfrak R^u_{\sec\ge 0} (N)$
the diffeomorphism type of the pair $(N,\ \textup{soul of}\ g)$ 
is locally constant.\newline
\textup{(iii)}
The diffeomorphism type of the pair
$(N,\ \textup{soul of}\ g)$ is constant
on connected components of $\mathfrak M^u_{\sec\ge 0} (N)$.
\end{thm}

Theorem~\ref{intro-thm: ambient isotopy} 
also holds for $\mathfrak M_{\sec\ge 0}^{c} (N)$ provided any
two metrics of $\sec\ge 0$ on $N$ have souls that intersect.

Thus to detect different connected components of 
$\mathfrak M_{\sec\ge 0}^u (N)$ it is enough to produce nonnegatively
curved metrics on $N$ such that no self-diffeomorphism of $N$
can move their souls to each other. From
Theorem~\ref{intro-thm: codim 4 inf many souls} we deduce:

\begin{cor}
For any integers $k\ge 3$, $m\ge 4$ the space
$\mathfrak M_{\sec\ge 0}^{u} (S^4\times S^k\times \mathbb R^m)$ 
has infinitely many connected components that lie in 
the same 
component of $\mathfrak M^{u} (S^4\times S^k\times \mathbb R^m)$. 
\end{cor}

Similarly,  Theorem~\ref{intro-thm: kpt modified} yields 
an infinite sequence of metrics that lie in different connected components
of $\mathfrak M_{\sec\ge 0}^{c} (N)$ and in the same component
of $\mathfrak M^{u} (N)$.

Even if the souls are diffeomorphic they need not be ambiently 
isotopic, as is illustrated 
by the following theorem exploiting examples of smooth knots 
due to Levine~\cite{Lev}.

\begin{thm}\label{intro-thm: knots mod space}
If $L$ is a closed manifold of $\sec\ge 0$, 
then $N:=S^7\times L\times\mathbb R^4$ admits metrics
that lie in different connected components of 
$\mathfrak M_{\sec\ge 0}^u (N)$ and in the same 
component of $\mathfrak M^{u} (N)$, and such
that their souls are diffeomorphic to $S^7\times L$
and not ambiently isotopic in $N$.
\end{thm}

Here $L$ is allowed to have dimension $0$ or $1$, and in general,
throughout the paper we treat $S^1$, $\mathbb R$, and a point as 
manifolds of $\sec\ge 0$.

\begin{ex}{\it
For $L=S^5$, note that
any closed manifold in the homotopy type of $S^7\times S^5$
is diffeomorphic to $S^7\times S^5$; in fact 
the structure set
of $S^7\times S^5$ fits into the surgery exact sequence between
the trivial groups $\Theta_{12}$ and $\pi_7(F/O)\oplus\pi_5(F/O)$
\textup{\cite[Theorem 1.5]{Cro}}.
Thus any soul in $S^7\times S^5\times \mathbb R^4$
is diffeomorphic to $S^7\times S^5$, while 
Theorem~\ref{intro-thm: knots mod space} detects
different components of the moduli space.}
\end{ex}

As mentioned above, there exist exotic $7$-spheres with $\sec\ge 0$
that appear as codimension $3$ souls in $S^7\times\mathbb R^3$.
Examples with non-diffeomorphic
simply-connected souls of codimension $2$ seems considerably 
harder to produce, as is suggested by the following:

\begin{thm}\label{thm: simply-conn and codim 2}
If a simply-connected manifold
$N$ admits complete nonnegatively curved metrics with souls
$S$, $S^\prime$ of dimension $\ge 5$ and codimension $2$,
then 
$S^\prime$ is diffeomorphic to the connected sum of $S$ with 
a homotopy sphere. 
\end{thm}

Thus non-diffeomorphic codimension two simply-connected
souls are necessarily homeomorphic, while until now non-diffeomorphic 
homeomorphic closed manifold of $\sec\ge 0$ have only been
known in dimension $7$, see e.g~\cite{KreSto, GroZil}.
In the companion paper~\cite{BKS-mod2} we show that
for every integer $r\ge 2$ there is an open $(4r+1)$-dimensional
simply-connected manifold that admits two metrics with
non-diffeomorphic codimension $2$ souls.

Non-diffeomorphic simply-connected souls
do not exist in codimension $1$, except possibly when
the soul has dimension $4$; indeed, any two codimension
$1$ simply-connected souls are h-cobordant, and hence diffeomorphic
provided their dimension is $\neq 4$. 
By contrast, manifolds with nontrivial fundamental group
may contain non-homeomorphic codimension $1$ souls:

\begin{ex} (\cite{Mil-haup})\label{ex: mil-lens} 
{\it
Let $L, L^\prime$ be homotopy equivalent, non-homeomorphic
$3$-dimensional lens spaces, such as $L(7,1)$, $L(7,2)$.
Then $L\times S^{2k}$, $L^\prime\times S^{2k}$
are non-homeomorphic and h-cobordant for $k>0$, hence
they can be realized as non-homeomorphic souls in
$N:=L\times S^{2k}\times\mathbb R$, which is
diffeomorphic to $L^\prime\times S^{2k}\times\mathbb R$.
In particular, $\mathfrak M_{\sec\ge 0}^c (N)$ is not connected.}
\end{ex}

Codimension $1$ case is special both for geometric and topological
reasons. As we show in Proposition~\ref{prop: codim one souls},
if a manifold $N$ admits a metric with a codimension $1$ soul,
then the obvious map 
$\mathfrak M_{\sec\ge 0}^{\infty,u} (N)\to\mathfrak M_{\sec\ge 0}^{\infty,c} (N)$ 
is a homeomorphism, and either space is homeomorphic to the disjoint
union of the moduli spaces of all possible pairwise non-diffeomorphic souls
of metrics in $\mathfrak M^\infty_{\sec\ge 0} (N)$.

Kreck-Stolz~\cite{KreSto} used index-theoretic arguments
to construct a closed simply-connected $7$-manifold
$B$ which carries infinitely many
metrics of $\Ric>0$ that lie in different components of
$\mathfrak M^\infty_{\mathrm{scal}> 0} (B)$. 
It it was shown in~\cite{KPT} 
that some other metrics on $B$ have $\sec\ge 0$ and
lie in infinitely many different components of
$\mathfrak M^\infty_{\mathrm{scal}\ge 0} (B)$. 
In particular, we conclude

\begin{cor} 
\label{cor: Kreck-Stolz}
$\mathfrak M^{\infty, c}_{\sec\ge 0} (B\times\mathbb R)$
has infinitely many connected components.
\end{cor}

We also give examples of
infinitely many isometric metrics that cannot be deformed 
to each other in through complete metrics of $\sec\ge 0$.

\begin{thm} \label{intro-thm: knots space of metrics}
If $n=4r-1$ and $3\le k\le 2r+1$
for some $r\ge 2$, then 
$\mathfrak R^u_{\sec\ge 0} (S^n\times\mathbb R^k)$
has infinitely many components that lie in the same component
of $\mathfrak R^u(S^n\times\mathbb R^k)$.
\end{thm}

\begin{thm}\label{intro-thm: space of metrics h-cob}
$\mathfrak R^u_{\sec\ge 0} (N)$ has infinitely many components if\newline
\textup{(i)}
$N=L\times L(4r+1,1)\times S^{2k}\times\mathbb R$ 
where $L$ is any complete manifold
of $\sec\ge 0$ and nonzero Euler characteristic, and
$k\ge 3$, $r>0$, \newline
\textup{(ii)} $N=M\times\mathbb R$ where $M$ is a closed oriented manifold
of even dimension $\ge 5$ with $\sec\ge 0$ such that
$G=\pi_1(M)$ is finite and $\mathrm{Wh}(G)$ is infinite.
\end{thm}

The proof of (i) relies on a geometric ingredient of independent
interest: if $S$, $S^\prime$ are souls of metrics lying in the same
component of $\mathfrak R^u_{\sec\ge 0} (N)$, then
the restriction to $S$ of any deformation retraction $N\to S^\prime$
is homotopic to a diffeomorphism; e.g. this applies to the
Sharafutdinov retraction.

{\bf Structure of the paper.} Section~\ref{sec: mod sp} contains various geometric results on the spaces of nonnegatively curved metrics, including
Theorem~\ref{intro-thm: ambient isotopy}.
Section~\ref{sec: souls of codim 4} contains proofs of
Theorems~\ref{intro-thm: codim 4 inf many souls}--\ref{intro-thm: kpt modified}
giving examples with infinitely many codimension $4$ souls. 
Sections~\ref{sec: vb with diffeo tot spaces}--\ref{sec: bundles of rank <4} 
is a topological study of vector bundles with diffeomorphic 
total spaces, especially those of fiber dimension $\le 3$;
in particular, there we prove 
Theorems~\ref{intro-thm: codim <4 finiteness}
and~\ref{thm: simply-conn and codim 2}.
In Section~\ref{sec: smooth knots} we prove
Theorem~\ref{intro-thm: knots mod space} implying that 
diffeomorphic souls they need not be ambiently 
isotopic; Theorem~\ref{intro-thm: knots space of metrics}
is also proved there. Section~\ref{sec: h-cob} contains the proof of
Theorem~\ref{intro-thm: space of metrics h-cob} and the geometric
result stated in the previous paragraph.

{\bf On topological prerequisites.} 
This paper employs a variety of topological tools,
which we feel are best learned from the following sources. 
We refer to books by Husemoller~\cite{Hus-fibr-bundl},
Milnor-Stasheff~\cite{MS-char-cl}, and 
Spanier~\cite[Chapter 6.10]{Spa-alg-top} 
for bundle theory and characteristic classes,
to books by Cohen~\cite{Coh}, and Oliver~\cite{Oli}, 
and Milnor's survey~\cite{Mil-wh} 
for Whitehead torsion and h-cobordisms, and
to monographs of Wall~\cite{Wal-book} and Ranicki~\cite{Ran-book}
for surgery theory. In the companion paper~\cite[Sections 3 and 8]{BKS-mod2}
we survey aspects of surgery that are most relevant to~\cite{BKS-mod2}
and to the present paper. 
We refer to~\cite[Section 9.2 and Proposition 9.20]{Ran-book} 
for results on classifying spaces for spherical fibrations
associated with topological monoids $F_k$, $G_k$, $F$, and
their identity components $SF_k$, $SG_k$, $SF$; 
observe that Ranicki
denotes $F_k$, $G_k$ by $F(k+1)$, $G(k)$, respectively. 

\section{Moduli spaces and souls}
\label{sec: mod sp}

In this section we prove Theorem~\ref{intro-thm: ambient isotopy},
Corollary~\ref{cor: Kreck-Stolz}, and related results.  
We focus on moduli spaces with uniform topology; 
Remark~\ref{rmk: normal euler cl} discusses when the same
results hold for moduli spaces with topology of convergence on compact subsets.
Here and elsewhere in the paper 
we follow notational Convention~\ref{intro-convention}.

Riemannian metrics are sections of a tensor bundle, so
they lie in a continuous function space, which is metrizable;
thus $\mathfrak R^u_{\sec\ge 0} (N)$ is metrizable, and in particular,
a map with domain $\mathfrak R^u_{\sec\ge 0} (N)$ is continuous
if and only if it sends convergent sequences to convergent sequences.

Theorem~\ref{intro-thm: ambient isotopy} follows immediately
from Lemma~\ref{lem: souls amb isot} below.
Indeed Lemma~\ref{lem: souls amb isot} implies that
the map sending $g$ in $\mathfrak R^u_{\sec\ge 0} (N)$ 
to the diffeomorphism class of the pair $(N,\ \textup{soul of}\ g)$ 
is locally constant, and hence continuous with respect to 
the discrete topology on the codomain, which implies that
it descends to a continuous map from the quotient space
$\mathfrak M^u_{\sec\ge 0} (N)$.

Let $S_i$ be a soul of $g_i$ in $\mathfrak R^u_{\sec\ge 0} (N)$,
and let $p_i\co N\to S_i$ denote the Sharafutdinov retraction, 
$\check g_i$ be the  induced metric on $S_i$. Since $S_i$ is convex, 
$\check g_i$ and $g_i$ induce the same
distance functions on $S_i$, which is denoted $d_i$.
For brevity $g_0$, $S_0$, $p_0$, $\check g_0$, $d_0$
are denoted by $g$, $S$, $p$, $\check g$, $d$, respectively.

\begin{lem} \label{lem: souls amb isot}
If $g_i$ converges to $g$ in 
$\mathfrak R^u_{\sec\ge 0} (N)$, then for all large $i$ \newline
\textup{(1)} 
$p_i\vert_{S}\co S\to S_i$ is a diffeomorphism,
\newline
\textup{(2)} 
the pullback metrics $(p_i\vert_{S})^\ast \check g_i$
converge to $\check g$ 
in $\mathfrak R^{0,u}_{\sec\ge 0} (S)$,\newline
\textup{(3)} 
$S_i$ is $C^\infty$ ambiently isotopic to $S$ in $N$.
\end{lem}

\begin{proof}
(1) Arguing by contradiction pass to a subsequence for which
$p_i\vert_{S}$ is never a diffeomorphism.
Wilking proved in~\cite{Wil} that Sharafutdinov retractions are 
smooth Riemannian submersions onto the soul. 
Note that $p_i(S)=S_i$ and $p(S_i)=S$
because degree one maps are onto.

Since the convergence $g_i\to g$ is
uniform, given any positive $\e, R$ and all large enough $i$
the distance functions 
$d_i$ are $\e$-close to $d$ on any $R$-ball in $(N, d)$.
Then $S$ has uniformly bounded $d_i$-diameter, and since
$p_i$ are distance non-increasing, and $p_i(S)=S_i$,
we conclude that $S_i$ has uniformly bounded $d_i$-diameter;
thus the metrics $d, d_i$ are close on $S$, and on $S_i$.
As $p_i$, $p$ are distance-nonincreasing,
with respect to $d_i$, $d$, respectively, the self-map 
$f_i:={p}{\tinycirc} {p_i}\vert_S$ of $(S, d)$ 
is almost distance non-increasing. Then compactness of $S$
implies via Ascoli's theorem that $f_i$ subconverges  
to a self-map of $(S, d)$, which is distance non-increasing and
surjective, and hence is an isometry.

Any isometry is a diffeomorphism. 
Diffeomorphisms form an open subset among smooth mappings, so
$p{\tinycirc} p_i\vert_S$ is a diffeomorphism for 
large $i$. It follows that $p_i\vert_S$ is an injective immersion,
and hence a diffeomorphism, as $S$ is a closed manifold,
giving a contradiction that proves (1).

(2) Arguing by contradiction, pass to a subsequence for which
$p_i^\ast\check g_i$ lies outside a $C^0$-neighborhood of $\check g$. 
Note that $p_i\co (S,d)\to (S_i,d_i)$ 
is a Gromov-Hausdorff approximation, indeed, if $x,y\in S$, then 
$d(x,y)$ is almost equal to
$d(f_i(x), f_i(y))\le d(p_i(x), p_i(y))$, where the right
hand side is almost equal to $d_i(p_i(x), p_i(y))$ which is
$\le d_i(x,y)$, which is almost equal to $d(x,y)$;
thus all the inequalities are almost equalities and hence
$d(x,y)$ is almost equal to $d_i(p_i(x), p_i(y))$.

By Yamaguchi's fibration theorem~\cite{Yam} there is a diffeomorphism
$h_i\co S_i\to S$ such that $h_i^\ast\check  g$ is $C^0$-close
to $\check g_i$. Note that $h_i\tinycirc p_i$ almost preserves $\check d$
so it subconverges to an isometry of $(S,\check g)$, and in particular
it pulls $\check g$ back to a metric that is $C^0$-close to $\check g$.
It follows that $p_i^\ast\check  g_i$ is $C^0$-close to $\check g$,
giving a contradiction which proves (2).

(3) 
Let $E(\nu_i)$ denote 
the total space of the normal bundle  $\nu_i$ to $S_i$.
Wilking showed in~\cite[Corollary 7]{Wil}
that there exists a diffeomorphism $e_i\co E(\nu_i)\to N$
such that $p_i\tinycirc e_i$ is the projection of $\nu_i$.
Thus (1) implies that the projection of $\nu_i$
restricts to a diffeomorphism from $e_i^{-1}(S)$ onto $S_i$, 
whose inverse is a section of $\nu_i$ with image $e_i^{-1}(S)$.
Any two sections of a vector bundle are ambiently isotopic,
so applying $e_i$ we get an ambient isotopy of $S$ and $S_i$ 
in $N$.
\end{proof}

\begin{rmk} 
\label{rmk: normal euler cl}
We do not know whether the conclusion of Lemma~\ref{lem: souls amb isot}
holds for $\mathfrak R^c_{\sec\ge 0} (N)$. The proof
of Lemma~\ref{lem: souls amb isot} works for 
$\mathfrak R^c_{\sec\ge 0} (N)$ as written provided $\mathrm{dist}(S, S_i)$ 
is uniformly bounded. This happens if any two metrics
in $\mathfrak R^c_{\sec\ge 0} (N)$ have souls that intersect
which as noted in~\cite{KPT} is true e.g.
when $N$ contains a soul with nontrivial 
normal Euler class. Note that except for examples 
discussed in Remark~\ref{rmk: kpt with nonzero Euler class} 
all the metrics we construct in this paper have souls with 
trivial normal Euler class.
\end{rmk}

\begin{rmk}
Let $\{S_i\}_{i\in I}$ be a collection of pairwise nondiffeomorphic 
manifolds
representing the diffeomorphism classes of souls of all possible
complete nonnegatively curved metrics on an open manifold
$N$, and for
$g\in \mathfrak R^u_{\sec\ge 0} (N)$, let $i(g)\in I$
be such that $S_{i(g)}$ is diffeomorphic to a soul of $(N,g)$.
By~\cite{Sha-souls} one has a well-defined map
that associates to $g$ the isometry class of its soul in
$\mathfrak M_{\sec\ge 0} (S_{i(g)})$, which can be thought of as a map 
$\mathfrak R^u_{\sec\ge 0} (N)\to
\coprod_i\mathfrak M_{\sec\ge 0} (S_{i})$, where 
the codomain is given the topology of disjoint union of 
$\mathfrak M_{\sec\ge 0} (S_{i})$'s. This map descends to
a map 
\[
\text{\bf soul}\co \mathfrak M^{k,u}_{\sec\ge 0} (N)\to
\coprod_i\mathfrak M^{k}_{\sec\ge 0} (S_{i}).
\] 
If $k=0$, then part (2) of Lemma~\ref{lem: souls amb isot}
implies that the map {\bf soul} is continuous (the continuity can be
checked on sequences in $\mathfrak R^{0,u}_{\sec\ge 0} (N)$ because
it is metrizable). 
\end{rmk}

\begin{rmk}\label{rmk: retraction of moduli spaces}
Suppose that $N$ has a soul $S$ with trivial normal bundle.
Let $\mathfrak M^{k,u}_{\sec\ge 0} (N, S)$ denote the
union of the components of $\mathfrak M^{k,u}_{\sec\ge 0} (N)$
consisting of the isometry classes of metrics with 
soul diffeomorphic to $S$. 
Then the map {\bf soul} restricts to a retraction 
\[
\mathfrak M^{0,u}_{\sec\ge 0} (N, S)\to\mathfrak M^{0}_{\sec\ge 0} (S)
\]
where $\mathfrak M^0_{\sec\ge 0} (S)$ sits in 
$\mathfrak M^0_{\sec\ge 0} (N, S)$ as
the set of isometry classes of Riemannian products
of nonnegatively curved metrics on $S$ and the standard
$\mathbb R^n$. Like any retraction
it induces a surjective maps on homotopy and homology, and hence
one potentially could get lower bounds on the topology of 
$\mathfrak M^{0,u}_{\sec\ge 0} (S\times\mathbb R^n)$ in terms of
topology of $\mathfrak M^0_{\sec\ge 0} (S)$.
Unfortunately, nothing is known about
the topology of $\mathfrak M^0_{\sec\ge 0} (S)$, which
naturally leads to the following.
\end{rmk}

\begin{prob} Find a closed manifold $S$ with non-connected
$\mathfrak M^0_{\sec\ge 0}(S)$.
\end{prob}

\begin{prob} Is the map 
$\text{\bf soul}\co \mathfrak M^\infty_{\sec\ge 0} (N)\to
\coprod_i\mathfrak M^\infty_{\sec\ge 0} (S_{i})$
continuous?
\end{prob}

The only known examples with 
non-connected $\mathfrak M^k_{\sec\ge 0}(S)$ are 
(modifications of) those
in~\cite{KreSto} where $k=\infty$ (it may suffice
to take $k$ sufficiently large but definitely not $k=0$). 
These examples were modified 
in~\cite{KPT} to yield a closed simply-connected
manifold $B$ admitting infinitely many metrics $g_i$ with 
$\sec\ge 0$ and $\mathrm{Ric}>0$ 
that lie in different components of 
$\mathfrak M^\infty_{\mathrm{scal}> 0}(B)$. 
It was asserted in~\cite{KPT} that $g_i$'s lie in
different components of $\mathfrak M^\infty_{\mathrm{scal}\ge 0}(B)$,
but it takes an additional argument which 
hopefully will be written by the authors of~\cite{KPT}. 
The following shows that $g_i$ lie in different
components of $\mathfrak M^\infty_{\sec\ge 0}(B)$. 

\begin{prop}\label{prop: via ricci flow}
Metrics of $\sec\ge 0$ and $\mathrm{Ric}>0$ 
on a closed manifold $X$ that lie in different components of 
$\mathfrak M^\infty_{\mathrm{Ric}> 0}(X)$
also lie in different components of 
$\mathfrak M^\infty_{\sec\ge 0}(X)$. 
\end{prop}
\begin{proof}
We abuse terminology by not distinguishing a metric
from its isometry class. First we show that each 
$h\in\mathfrak M^\infty_{\sec\ge 0}(X)$ has a neighborhood $U_h$
such that any $h^\prime\in U_h$ can be joined to $h$ by a path
of metrics $h_s$ with $h_0=h$, $h_1=h^\prime$
and $\mathrm{Ric}(h_s)>0$ for $0<s<1$.
If there is no such $U_h$, then using Ebin's slice
theorem~\cite{Ebi} one can show that there is a sequence 
$h_i\in\mathfrak R^\infty_{\sec\ge 0}(X)$ converging to $h$
such that $h_i$ cannot be joined to $h$ by a path as above.
B{\"ohm}-Wilking~\cite{BohWil} showed that Ricci flow
instantly makes a metric of $\sec\ge 0$ on a closed manifold with
finite fundamental group into a metric with $\mathrm{Ric}>0$.
Thus $h_i$ and $h$ can be flown to nearby metrics $h_i(t)$, $h(t)$
of positive Ricci curvature where $h_i(t)\to h(t)$ for any fixed
small $t$. Since $\mathfrak M^\infty_{\mathrm{Ric}> 0}(X)$ is open
in the space of all metrics, Ebin's slice theorem ensures that
if $i$ is large enough $h_i(t)$, $h(t)$ can be joined by a path in  
$\mathfrak M^\infty_{\mathrm{Ric}> 0}(X)$, and concatenating 
the three paths yields a desired path from $h_i$ to $h$
via $h_i(t)$ and $h(t)$.

Given an open cover $\{U_k\}$ of a connected
set for any two $g, g^\prime$ in this set there exists a finite sequence
$g_0=g, g_1,\dots , g_n=g^\prime$ such that $g_k\in U_k$ and 
$U_k\cap U_{k-1}\neq\emptyset$ for every 
$0<k\le n$~\cite[Section 46, Theorem 8]{Kur} . 

Thus given two metrics $g, g^\prime$ in a component of
$\mathfrak M^\infty_{\sec\ge 0}(X)$, we get a finite
sequence $g_k$ in this component with
$g_0=g, g_1,\dots , g_n=g^\prime$ and such that 
for each $k$ one can join $g_{k-1}$ to $g_k$ by a path
of metrics that have $\mathrm{Ric}>0$ except possibly at 
endpoints. By assumption, $g, g^\prime$ have $\mathrm{Ric}>0$.
By construction the paths backtrack at 
$g_k$, as they are given by Ricci flow $g_k(t)$ near $g_k$, 
so the concatenated path from $g$ to $g^\prime$
can be cut short at $g_1,\dots g_{n-1}$ 
to entirely consist of metrics of $\mathrm{Ric}>0$.
\end{proof} 

Now Corollary~\ref{cor: Kreck-Stolz} follows from the
proposition below applied for $k=\infty$.

\begin{prop}\label{prop: codim one souls} 
For an integer $k\ge 0$, set $k_0$ to be the maximum of $0$ and $k-1$;
for $k=\infty$, set $k_0=\infty$.
If $N$ admits a complete metric with $\sec\ge 0$ and 
a codimension $1$ soul, then the maps
$\text{\bf soul}\co \mathfrak M^{k,c}_{\sec\ge 0} (N)\to
\coprod_i\mathfrak M^{k_0} (S_{i})$ and 
$\mathrm{id}\co\mathfrak M_{\sec\ge 0}^{k,u}(N)\to \mathfrak M_{\sec\ge 0}^{k,c}(N)$ 
are homeomorphisms.
\end{prop}

{\em Added in 2018:} The proof below only works if $k=k_0$, i.e., $k$ is zero or infinity. 
If $k\neq k_0$, the proof shows that ${\bf soul}$ is continuous, but
its inverse is clearly discontinuous. 
A correction for $k\ge 2$ can be found 
in~\cite[Corollary 2.2]{TusWie}.

\begin{proof}
It suffices to show that the maps
$\text{\bf soul}\co \mathfrak M^{k,c}_{\sec\ge 0} (N)\to
\coprod_i\mathfrak M^{k_0} (S_{i})$
and $\text{\bf soul}\co \mathfrak M^{k,u}_{\sec\ge 0} (N)\to
\coprod_i\mathfrak M^{k_0} (S_{i})$ are homeomorphisms,
and the argument below works for both maps.

By the splitting theorem any complete metric $g$ of $\sec\ge 0$
on $N$ locally splits off an $\mathbb R$-factor that is 
orthogonal to the soul. The splitting becomes
global in the cover of order $\le 2$
that corresponds to the first Stiefel-Whitney class $w_1$
of the normal bundle to the soul.
If $w_1=0$, then $g$ is the product of $\mathbb R$
and a closed nonnegatively curved manifold $S_g$, in which case there 
is a unique soul $S_g\times\{t\}$ through every point.
If $w_1\neq 0$, and if $S_g$ is a soul of $g$,
then $g$ can be written as
$\tilde S_g\times_{O(1)}\mathbb R$, where $\tilde S_g$ is the $2$-fold
cover of $S_g$ induced by $w_1$, and 
$S_g=\tilde S_g\times_{O(1)}\{0\}$; in this case
$S_g$ is a unique soul of $g$ because by~\cite{Yim-souls}
any soul is obtained by exponentiating
a parallel normal vector field along $S_g$, 
but the existence of a nonzero parallel normal vector field along
$S_g$ would imply triviality of the normal line
bundle to $S_g$ contradicting $w_1\neq 0$.

As follows e.g. from Section~\ref{sec: vb with diffeo tot spaces},
if two real line bundles 
over closed manifolds $S$, $S^\prime$ have diffeomorphic total spaces,
then the line bundles are either both trivial, or both nontrivial.
Thus triviality of $w_1$ depends only on $N$, and not on the metric.

If $w_1=0$, then 
the map {\bf soul} has a continuous inverse induced by
the map that sends a metric on the soul to its product 
with $\mathbb R$.

If $w_1\neq 0$, then each closed nonnegatively curved manifold $S$
that is homotopy equivalent to $N$ has a 
$2$-fold cover $\tilde S$ induced by $w_1$. Thus a metric
in $\mathfrak M_{\sec\ge 0} (S)$ gives rise 
to the metric $\tilde S\times_{O(1)}\mathbb R$, which
defines a continuous inverse for {\bf soul}. 

Finally, we show that {\bf soul} is continuous. It suffices to do so for
for the topology of convergence on compact sets as
$\mathrm{id}\co\mathfrak M_{\sec\ge 0}^{k,u}(N)\to \mathfrak M_{\sec\ge 0}^{k,c}(N)$
is continuous.
Let $g_j\to g$ be a converging sequence in
$\mathfrak R_{\sec\ge 0}^{k,c}(N)$. By Lemma~\ref{lem: souls amb isot}
their souls of $g_j$ converge in $C^0$ topology (as abstract Riemaniann
manifolds) to the soul of $g$.
This gives continuity of {\bf soul} for $k\in \{0,1\}$, 
so we assume $k\ge 2$.

Fix a soul $S$ of $g$. Then there exists a soul
$S_j$ of $g_j$ that intersects $S$; 
indeed, given a complete metric $g^\prime$ of $\sec\ge 0$ on $N$,
if $w_1=0$, then $g^\prime$ has a soul through every point
of $N$, and if $w_1\neq 0$, then souls of $g^\prime$ and $g$
must intersect because $w_1$ can be interpreted as the first obstruction 
to deforming the homotopy equivalence $S_{g^\prime}\to S_g$ away from $S_g$.

As $S_j$ (abstractly) $C^0$ converge to $S$, their diameters are
uniformly bounded, in particular, they all lie
in a compact domain $D$ of $N$. 
Convergence $g_j\to g$ in $C^k$ topology implies convergence
$\nabla_{g_j}^{l} R_{g_j}\to\nabla_g^{l} R_g$
of covariant derivatives of the curvature tensors
for every $l\le k-2$, and in particular,
one gets a uniform bound on $\norm{\nabla_{g_j}^{l} R_{g_j}}$
over $D$.
Since $S_j$ is totally geodesic, the restriction of
$\nabla_{g_j}^{l} R_{g_j}$ to $S_j$ is 
$\nabla_{g_j\vert_{S_j}}^{l} R_{g_j\vert_{S_j}}$~\cite[Proposition 8.6]{KN},
so $\norm{\nabla_{g_j\vert_{S_j}}^{l} R_{g_j\vert_{S_j}}}$
are uniformly bounded for $l\le k-2$. Since the $C^0$ limit $S$ of $S_j$
has the same dimension, the convergence is without collapse, so
there is a common lower injectivity radius bound for $S_j$.
Hence the family $S_j$ is precompact in $C^{k-1}$ topology~\cite[page 192]{Pet-conv},
but since $S_j$ converges to $S$ in $C^0$ topology, all $C^{k-1}$ limit
points of $S_j$ are isometric to $S$ because Gromov-Hausdorff limits are 
unique up to isometry. Thus $S_j$ converges to $S$ in $C^{k-1}$ topology,
as claimed.
\end{proof}

\begin{quest}
Can $k_0$ in Proposition~\ref{prop: codim one souls}
be replaced by $k$?
\end{quest}

\begin{rmk} An analog of Proposition~\ref{prop: codim one souls} holds
for complete $n$-manifolds of $\Ric\ge 0$ with nontrivial
$(n-1)$-homology, because each such manifolds is a flat line bundle
over a compact totally geodesic submanifold~\cite{SorShe}.
In particular, once it is shown that metrics $g_i$ 
of~\cite{KPT} lie in
different components of $\mathfrak M^\infty_{\mathrm{scal}\ge 0}(B)$,
we can conclude that
$\mathfrak M^\infty_{\mathrm{Ric}\ge 0} (B\times\mathbb R)$
has infinitely many connected components.
\end{rmk}

As mentioned in the introduction, $\mathfrak M^u (N)$
need not be connected. It is therefore desirable to arrange 
our metrics with non-diffeomorphic souls to lie in the same
component of $\mathfrak M^u (N)$. This can be accomplished under
a mild topological assumption:

\begin{prop}\label{prop: guijarro modification}
Suppose an open manifold $N$ admits two
complete metrics of $\sec\ge 0$ with souls $S$, $S^\prime$.
If the normal sphere bundle to $S$ is simply-connected and has
dimension $\ge 5$,
then  $N$ admits two 
complete metrics of $\sec\ge 0$ with souls $S$, $S^\prime$
which lie in the same path-component of $\mathfrak R^u (N)$.
\end{prop}
\begin{proof}
By Proposition~\ref{prop: normal bundles are fiber homotopic}
the normal sphere bundle to $S^\prime$ is also simply-connected,
and by Lemma~\ref{lem: s1-bundles are hcobordant} if normal sphere bundles
to $S$, $S^\prime$ are chosen to be disjoint, then
the region between them is a (trivial) h-cobordism. 
Thus closed tubular neighborhoods of $S$, $S^\prime$ are diffeomorphic. 
The complement of an open tubular neighborhood of the soul
is of course the product of a ray and the boundary of the tubular neighborhood.
The diffeomorphism of closed tubular neighborhoods of $S$, $S^\prime$ 
extends to a self-diffeomorphism of $N$, which
can be chosen to preserve any given product structures on the complements
of tubular neighborhoods, and which is identity near $S$ and $S^\prime$.

By~\cite{Gui} any complete metric of $\sec\ge 0$ can be
modified by changing it outside a sufficiently small tubular 
neighborhood of the soul so that the new metric has the same 
soul and outside a larger tubular neighborhood it is the 
Riemannian product of a ray and a metric on 
the normal sphere bundle. 
Performing this modification
to the metrics at hand, and pulling back one of the metric
via a self-diffeomorphism of $N$ as above, one 
we get nonnegatively curved metrics $g$, $g^\prime$ 
with souls $S$, $S^\prime$ such that outside some of their 
common tubular neighborhood $D=D^\prime$ the metrics are 
Riemannian products $\d D\times\mathbb R_+$, 
$\d D^\prime\times\mathbb R_+$.  
with the same $\mathbb R_+$-factor. Now the convex combination of
$g$, $g^\prime$ defines a path joining $g$, $g^\prime$
in $\mathfrak R^u (N)$. 
\end{proof}

\section{Infinitely many souls of codimension $4$}
\label{sec: souls of codim 4}
 
This section contains examples of manifolds that admit
metrics with infinitely many non-homeomorphic souls of
codimension $4$. The examples are obtained by modifying
arguments in~\cite{Bel, KPT} and invoking the new topological 
ingredient, Proposition~\ref{prop: disk bundles pullback} below,
which is best stated with the following notation.

Given vector bundles $\a_0$, $\b_0$ over a space $Z$,
let ${\bf V}(Z, \a_0, \b_0)$ be the set of pairs 
$(\a, \b)$ of vector bundles over $Z$ such that 
$\a$, $\b$ are (unstably) fiber homotopy equivalent 
to $\a_0, \b_0$,
respectively, and the rational Pontryagin classes
of $\a\oplus\b$, $\a_0\oplus\b_0$ become equal when pullbacked
via the sphere bundle projection $b\co S(\b)\to Z$.
%
%
Also denote the fiber dimension of $S(\a_0)$, $S(\b_0)$ by
$k_{\a_0}$, $k_{\b_0}$, respectively. 

\begin{prop} \label{prop: disk bundles pullback}
If $k_{\a_0}+k_{\b_0}+\dim(Z)\ge 5$ and
$k_{\a_0}\ge 2$, and $Z$ is a closed smooth manifold,
then the number of diffeomorphism
types of the disk bundles $D(b^\#\a)$ with 
$(\a, \b)$ in ${\bf V}(Z, \a_0, \b_0)$  is finite.
\end{prop}

\begin{proof}
Denote the sphere bundle projection of $\a$, $\b$, $\a_0$, $\b_0$
by $a$, $b$, $a_0$, $b_0$, respectively, and fiber homotopy
equivalences by $f_\a\co S(\a)\to S(\a_0)$ and
$f_\b\co S(\b)\to S(\b_0)$.

The fiberwise cone construction yields a homotopy equivalence
\[
\hat f_\a\co (D(\a), S(\a))\to (D(\a_0), S(\a_0))
\]
that extends $f_\a$ and satisfies ${a_0}{\tinycirc}{\hat f_\a}=a$.
Pulling back $\hat f_\a$ via $b$ gives a homotopy equivalence
$b^\#\hat f_\a\co (D(b^\#\a), S(b^\#\a))\to (D(b^\#\a_0), S(b^\#\a_0))$.
Since $b={b_0}{\tinycirc}{f_\b}$,
the disk bundle $D(b^\#\a_0)$ is the $f_\b$-pullback of
$D(b_0^\#\a_0)$, so composing $b^\#\hat f_\a$ with 
the bundle isomorphism induced by $f_\b$ gives a homotopy
equivalence 
\[
F_{\a,\b}\co (D(b^\#\a), S(b^\#\a))\to 
(D(b_0^\#\a_0), S(b_0^\#\a_0)).
\]
Now we show that $F_{\a,\b}$ pulls back rational 
Pontryagin classes. The tangent bundles to 
$D(b^\#\a)$ and $D(b_0^\#\a_0)$ are determined by their
restrictions to the zero sections, 
and these restrictions stably are respectively
\[
b^\#\a\oplus\tau_{S(\b)}=
b^\#\a\oplus b^\#(\b\oplus\tau_Z)=b^\#(\a\oplus\b\oplus\tau_Z)
\] 
and $b_0^\#(\a_0\oplus\b_0\oplus \tau_Z)$. 
The restriction of $F_{\a,\b}$ to the zero section is $f_\b$, so pulling 
back the latter bundle via $f_\b$ gives 
two bundles over $S(\b)$, namely, 
$b^\#(\a\oplus\b\oplus \tau_Z)$ and 
$b^\#(\a_0\oplus\b_0\oplus\tau_Z)$, which
by assumption have the same rational total Pontryagin class.

Arguing by contradiction lets us pass to subsequences, thus
since rational Pontryagin classes determine a stable
vector bundle up to finite ambiguity, we may
pass to a subsequence in ${\bf V}(Z, \a_0, \b_0)$ 
for which the $F_{\a,\b}^{-1}$-pullbacks of all the bundles
$b^\#(\a\oplus\b\oplus\tau_Z)$ to $D(b_0^\#\a_0)$
are isomorphic. Fix $(\a_1,\b_1)$ in the subsequence so that 
$G_{\a,\b}:={F_{\a,\b}}{\tinycirc}{F_{\a_1,\b_1}^{-1}}$
is now tangential for any $(\a,\b)$.

To finish the proof we need a well-known tangential
surgery exact sequence 
\[
L_{n+1}^s(\pi_1(Y), \pi_1(\partial Y))~\longrightarrow~
{\bf S}^{s,t}(Y, \d Y))~\longrightarrow~[Y,F]~\longrightarrow~
L_{n}^s(\pi_1(Y), \pi_1(\partial Y))
\]
described e.g. in~\cite[Section 8]{BKS-mod2},
where ${\bf S}^{s,t}(Y, \d Y))$ is the tangential simple structure set
for a smooth manifold with boundary $Y$ of dimension $n\ge 6$. 

Set $Y:= D(b_1^\#\a_1)$; then
$G_{\a,\b}\co (D(b^\#\a), S(b^\#\a))\to (Y,\d Y)$
represents an element in ${\bf S}^{s,t}(Y, \d Y)$.
The assumption $k_{\a_0}\ge 2$
ensures that $\d Y\to Y$ is a $\pi_1$-isomorphism so that
the Wall groups $L_*^s(\pi_1(Y), \pi_1(\partial Y)$ vanish,
and the other dimension assumption gives 
$\dim(Y)=k_{\a_0}+k_{\b_0}+1+\dim(Z)\ge 6$.
By exactness ${\bf S}^{s,t}(Y, \d Y))$
is bijective to the set $[Y, F]$, which is a 
finite~\cite[Proposition 9.20(iv)]{Ran-book},
so that manifolds $D(b^\#\a)$ fall
into finitely many diffeomorphism classes. 
\end{proof}

\begin{rmk}
\label{rmk: fiberwise cone [Z, F]}
If in the definition of ${\bf V}(Z, \a_0, \b_0)$ 
we require that $\a\oplus\b$, $\a_0\oplus\b_0$ are stably
isomorphic, then the number of diffeomorphism types 
of manifolds $D(b^\#\a)$ is at most the order of the set
$[Z,F]$.
Indeed, let $\hat b\co D(\b)\to Z$, $\hat b_0\co D(\b_0)\to Z$
denote the disk bundle projections, and extend $F_{\a,\b}$ 
by the fiberwise cone construction to the homotopy
equivalence of triads
\[
\hat F_{\a,\b}\co D(\hat b^\#\a)\to D(\hat b_0^\#\a_0),
\]
which is tangential as $\a\oplus\b$, $\a_0\oplus\b_0$ are stably
isomorphic. Hence $F_{\a,\b}$ is also tangential, as a
restriction of $\hat F_{\a,\b}$ to submanifolds with 
trivial normal bundles. The geometric definition of
the normal invariant (see~\cite{Wal-book} after Lemma 10.6) 
easily implies that the normal invariant of $F_{\a,\b}$ is the
restriction of the normal invariant of $\hat F_{\a,\b}$, hence
the normal invariant of $F_{\a,\b}$ lies in the image 
of the restriction
$[D(\hat b_0^\#\a_0), F]\to [D(b_0^\#\a_0), F]$
whose domain is bijective to $[Z,F]$, 
as $D(\hat b_0^\#\a_0)$ is homotopy equivalent to $Z$. 
As in the proof of Proposition~\ref{prop: disk bundles pullback} 
the tangential surgery exact sequence implies that
the number of diffeomorphism types 
of manifolds $D(b^\#\a)$ is at most the number of 
normal invariants of the maps $F_{\a,\b}$, proving the claim.
\end{rmk}

Given $m\in \pi_4(BSO_3)\cong\mathbb Z$,
let $\xi_m^n$ be the corresponding rank $n$ vector bundle over $S^4$
with structure group $SO_3$ sitting in $SO_n$ 
in the standard way. Let $\eta_{l,m}^{k,n}$
denote the pullback of $\xi_m^n$ via the 
sphere bundle projection $S(\xi_l^k)\to S^4$.
Theorem~\ref{intro-thm: codim 4 inf many souls}
is obtained from the following by setting $l=0=m$.

\begin{thm} \label{thm: codim 4 souls over S4} 
If $k, n\ge 4$, then
$E(\eta_{l,m}^{k,n})$ admits infinitely many complete metrics 
of $\sec\ge 0$ with pairwise non-homeomorphic souls. 
\end{thm}
\begin{proof}
It is explained in~\cite{Bel}
that $S(\xi_l^k)$, $S(\xi_i^k)$ are fiber homotopy equivalent 
if $l-i$ is divisible by $12$ for $k\ge 4$.
(In fact, up to fiber homotopy equivalence there are only finitely 
many oriented $S^3$-fibrations over a finite complex $Z$
that admit a section, because their classifying map
in $[Z, BSG_4]$ factor through $BSF_3$
and $[Z, BSF_3]$ is finite as $BSF_3$ is
rationally contractible, see~\cite[Proposition 9.20(i)]{Ran-book}.)

Also it is noted in~\cite{Bel} that 
$\xi_l^k\oplus\xi_m^n$ and $\xi_i^k\oplus\xi_j^n$
are equal in $\pi_4(BSO)$ if $l+m=i+j$.
Of course if $j:=l-i+m$ and $l-i$ 
is divisible by $12$, then $m-j$ is divisible by $12$.

Thus we get an infinite family $(\xi_i^k, \xi_{l-i+m}^n)$ 
parametrized by $i$ with $l-i$ divisible by $12$ such that
$\xi_i^k\oplus\xi_{l-i+m}^n=\xi_l^k\oplus\xi_m^n$ in $BSO$,
and $S(\xi_i^k)$, $S(\xi_{l-i+m}^n)$ is fiber homotopy equivalent
to $S(\xi_l^k)$, $S(\xi_m^n)$, respectively.

By Proposition~\ref{prop: disk bundles pullback} 
$D(\eta_{i,l-i+m}^{k,n})$ lie in finitely many diffeomorphism
classes one of which must contain $D(\eta_{l,m}^{k,n})$. 
A priori this does not show that there are infinitely many 
$D(\eta_{i,l-i+m}^{k,n})$'s that are diffeomorphic to 
$D(\eta_{l,m}^{k,n})$. Yet $\pi_4(F)=0$ so 
Remark~\ref{rmk: fiberwise cone [Z, F]}
implies that $F_{\xi_i^k,\xi_{l-i+m}^n}\co 
D(\eta_{i,l-i+m}^{k,n})\to D(\eta_{l,m}^{k,n})$ 
is homotopic to a diffeomorphism. (Without invoking
Remark~\ref{rmk: fiberwise cone [Z, F]} we only get
an infinite sequence of $D(\eta_{i,l-i+m}^{k,n})$'s 
that are diffeomorphic to some
$D(\eta_{i_0,l-i_0+m}^{k,n})$.)

As in~\cite{Bel} results of Grove-Ziller
show that each $E(\eta_{i,l-i+m}^{k,n})$
are nonnegatively curved with zero section $S(\xi_i^k)$
being a soul, and $p_1(S(\xi_i^k))$ is $\pm 4i$-multiple of
the generator, so assuming $i\ge 0$, we get that the souls
are pairwise non-homeomorphic.
\end{proof}

The proof of Theorem~\ref{thm: kpt modified} below is a 
slight variation of an argument in~\cite{KPT}. A major difference
is in employing Proposition~\ref{prop: disk bundles pullback},
and checking it is applicable,
in place of ``above metastable range'' considerations of~\cite{KPT}. 
Another notable difference is that to satisfy the conditions of
Proposition~\ref{prop: disk bundles pullback} we have to vary 
$q,r$ and keep $a,b$ fixed, while exactly the opposite 
is done in~\cite{KPT}. This requires a number of minor changes, so
instead of extracting what we need from~\cite{KPT}
we find it easier (and more illuminating) to present a 
self-contained proof below; we stress that all computational tricks 
in the proof are lifted directly from~\cite{KPT}.

Recall that for a cell complex $Z$ each element in $H^2(Z)$
can be realized as the Euler class of a unique $SO_2$-bundle
over $Z$.

Let $X=S^2\times S^2\times S^2$. Fix an obvious basis 
in $H^2(X)$ whose elements are dual to the $S^2$-factors.
Let $\g$, $\xi$, $\mu$ be the complex line bundles
over $X$ with respective Euler classes 
$(a,b,0)$,  $(0, q,r)$, $(0,-q, r)$ in this basis,
where $a,b, q,r$ are nonzero integers and $a,b$ are coprime.
Let $\eta=\xi\oplus \e$ and $\z=\mu\oplus\e$,
where $\e$ is the trivial complex line bundle. 
 
Denote the pullback of $\eta$, $\z$ via 
$\pi_\g\co S(\g)\to X$ by $\hat\eta$, $\hat\z$, respectively, 
and the pullback of $\g$ via 
$\pi_\eta\co S(\eta)\to X$ by $\hat\g$.
By definition of pullback, 
$S(\hat\eta)$ and $S(\hat\g)$
have the same total space, which we denote $M_{\g,\eta}$.
Denote by $\pi_{\hat\g}$, $\pi_{\hat\eta}$ 
the respective sphere bundle projections
$S(\hat \eta)\to S(\eta)$, $S(\hat \g)\to S(\g)$;
note that ${\pi_{\hat\g}}{\tinycirc}{\pi_\eta}=
{\pi_{\hat\eta}}{\tinycirc}{\pi_\g}$.
Let $\tilde\z$ be the pullback of $\hat\z$ via 
$\pi_{\hat\eta}\co M_{\g,\eta}\to S(\g)$.
With these notations we prove:

\begin{thm}\label{thm: kpt modified}
\textup{(i)} 
For a universal $c>0$, the manifold
$E(\tilde\z)$ admits a complete metric
with $\sec(E(\tilde\z))\in [0,c]$ such that 
the zero section $M_{\g,\eta}$ of $\tilde\z$ is
a soul of diameter $1$.\newline
\textup{(ii)}
For fixed $\g$ and variable $\eta$, $\z$, the manifolds
$D(\tilde\z)$ lie into finitely many diffeomorphism classes, while
the manifolds $M_{\g,\eta}$ lie in  
infinitely many homeomorphism classes. 
\end{thm}

\begin{proof} \textup{(i)} 
Recall that
any principal $S^1$-bundle $P$ over $(S^2)^n$
can be represented as $(S^3)^n\times_\rho S^1$
where $\rho\co T^n\to S^1$ is a homomorphism and
$T^n$ acts on $(S^3)^n$ as the product of standard
$S^1$-actions on $S^3$. (Indeed, the pullback of 
the $S^1$-bundle to $(S^3)^n$ can be trivialized
as $H^2((S^3)^n)=0$, and $\rho$ comes from the $T^n$-action
on the $S^1$-factor.) Therefore, $P\to (S^2)^n$
can be identified with $(S^3)^n/\ker(\rho)\to (S^3)^n/T^n$.

Specializing to our situation, 
let $\rho_\g$, $\rho_\eta$, $\rho_\z$ be the
homomorphisms $T^3\to S^1$ corresponding to 
the principal circle bundles that are (uniquely) determined by
$\g,\eta,\z$, respectively. Thus
the principal circle bundle $S(\g)$ equals 
$(S^3)^3/\ker(\rho_\g)$, and 
the fiber product 
$S(\eta)\oplus\z$ can be written as the associated bundle
$(S^3)^3\times_{\rho_\eta} S^3\times_{\rho_\z} \mathbb R^4$;
this is an $S^3\times\mathbb R^4$-bundle over $B$.
The pullback of this latter bundle to $S(\g)$ has
total space $E(\tilde\z)$, and it
can then be written as 
$(S^3)^3\times_{\rho_\eta\vert_{\ker(\rho_\g)}} 
S^3\times_{\rho_{\g}\vert_{\ker(\rho_\g)}} \mathbb R^4$.

All the actions are isometric, so
giving $\mathbb R^4$ a rotationally symmetric metric isometric
to $S^3\times \mathbb R_+$ outside a compact subset, we see
that $E(\tilde\z)$
gets a Riemannian submersion metric of $\sec\in [0,c]$
for a universal $c$. By a standard argument involving
a rotationally symmetric exhaustion function on $\mathbb R^4$,
the zero section $M_{\g,\eta}$ is a soul.
Since $M_{\g,\eta}$ is a quotient of 
$(S^3)^3\times_{\rho_\eta} S^3$ that can be further
Riemannian submersed onto a fixed manifold $S(\g)$,
the diameter of $M_{\g,\eta}$ is uniformly bounded 
above and below which can be rescaled to be $1$, 
while keeping universal curvature bounds on $E(\tilde\z)$.

\textup{(ii)}
First, we show that $M_{\g,\eta}$ fall into infinitely many
homeomorphism types.
Since $\tau_X$ is stably trivial, 
computing the first Pontryagin class gives
\[
p_1(M_{\g,\eta})=\pi_{\hat\eta}^\ast\, p_1(\hat\eta\oplus\tau_{S(\g)})=
\pi_{\hat\eta}^\ast\pi_\g^\ast p_1(\eta\oplus\g\oplus \tau_X)=
\pi_{\hat\eta}^\ast\pi_\g^\ast p_1(\xi\oplus\g).
\] 
Now $p_1(\g\oplus\xi)=
p_1(\g)+p_1(\xi)=e(\g)^2+e(\xi)^2$ 
and the Gysin sequence for $\g$ gives $\pi_\g^\ast e(\g)^2=0$
because the kernel of 
$\pi_\g^\ast\co H^4(X)\to H^4(S(\g))$ is the 
image of the (cup) multiplication by $e(\g)$.

We compute $\pi_{\hat\eta}^\ast\pi_\g^\ast p_1(\xi)$
from the commutative diagram below, whose rows are 
Gysin sequences for 
$\g$, $\hat\g$, while all vertical arrows are isomorphisms for $i\le 2$
because they fit into the Gysin sequences for $\eta$, $\hat\eta$
where injectivity follows as  $e(\eta)$, $e(\hat\eta)$ vanish
and surjectivity holds for $i\le 2$ as $X$, $S(\eta)$, $M_{\g,\eta}$ 
are simply-connected, which uses that $a,b$ are coprime.
\[
\xymatrix{
H^i(X)\ar[r]^{\cup\,e(\g)}\ar[d]^{\pi_\eta^\ast}&
H^{i+2}(X)\ar[r]^{\pi_\g^\ast}\ar[d]^{\pi_\eta^\ast}&
H^{i+2}(S(\g))\ar[d]^{{\pi_{\hat\eta}}^\ast}\ar[r]&
H^{i+1}(X)=0\\
H^i(S(\eta))\ar[r]^{\cup\,e(\hat\g)}&
H^{i+2}(S(\eta))\ar[r]^{{\pi_{\hat\g}}^\ast}&
H^{i+2}(M_{\g,\eta})\ar[r]&0&
}
\]
Also the
commutativity of the rightmost square implies that ${\pi_{\hat\g}}^\ast$
is onto.

Let ${\bf x}$, ${\bf y}$, ${\bf z}$ be the basis in
$H^2(S(\eta))$ corresponding to the chosen basis in $H^2(X)=\mathbb Z^3$;
thus $\pi_\eta^\ast e(\xi)=q{\bf y}+r{\bf z}$, and
$e(\hat\g)=a{\bf x}+b{\bf y}$ which
is primitive as $a,b$ are coprime. 
Another basis in $H^2(S(\eta))$ is $a{\bf x}+b{\bf y}$, 
$-m{\bf x}+n{\bf y}$, ${\bf z}$ where $n,m$ are integers with $an+bm=1$.
Thus 
$H^2(M_{\g,\eta})$ is isomorphic to $\mathbb Z^2$ generated by 
${\bf u}:=\pi_{\hat\g}^\ast ({\bf z})$ and
${\bf w}:=\pi_{\hat\g}^\ast (-m{\bf x}+n{\bf y})$. In particular,
$\pi_{\hat\g}^\ast$ maps ${\bf y}$ to $a{\bf w}$
because $-am{\bf x}+an{\bf y}={\bf y}-m(a{\bf x}+b{\bf y})$,
and similarly $\pi_{\hat\g}^\ast({\bf x})=-b{\bf w}$, even though we
do not use it.

The cup squares of ${\bf x}^2$, ${\bf y}^2$, ${\bf z}^2$ vanish
because the $S^2$-factors of $X$ have trivial self-intersection
numbers when computed in some $S^2\times S^2$-factor of $X$.
Now $\pi_\eta^\ast e(\xi)=q{\bf y}+r{\bf z}$ implies
$\pi_\eta^\ast p_1(\xi)=\pi_\eta^\ast e(\xi)^2=(2qr){\bf y}{\bf z}$, hence 
\[
p_1(M_{\g,\eta})=\pi_{\hat\eta}^\ast\pi_\g^\ast p_1(\xi)=
\pi_{\hat\g}^\ast\pi_\eta^\ast\,p_1(\xi)=
(2qra){\bf w}{\bf u}.
\]
The basis ${\bf z}(a{\bf x}+b{\bf y})$, 
${\bf z}(-m{\bf x}+n{\bf y})$, 
${\bf x}{\bf y}=(a{\bf x}+b{\bf y})(n{\bf y}+m{\bf x})$ 
in $H^4(S(\eta))$ is projected to $0$, ${\bf w}{\bf u}$, $0$
by ${\hat\pi_\g}^\ast$, in particular,
${\bf w}{\bf u}$ generates $H^{4}(M_{\g,\eta})$.
It follows that for any fixed $a,b$ by varying $q$, $r$,
we get (by the topological invariance of rational
Pontryagin classes) that the manifolds $M_{\g,\eta}$ lie in
infinitely many homeomorphism types.

We show that the manifolds $D(\tilde\z)$ lie in 
finitely many diffeomorphism types by applying 
Proposition~\ref{prop: disk bundles pullback}
for $(\a,\b)=(\hat\z,\hat\eta)$. 
To see it applies note that 
$p_1(\z\oplus\eta)=p_1(\mu)+p_1(\xi)=e(\mu)^2+e(\xi)^2$, so
$p_1(\hat\z\oplus\hat\eta)=(-2qr+2qr){\bf y}{\bf z}=0$. 
It remains to check that $S(\hat\eta)$, $S(\hat\z)$ 
lie in finitely many fiber homotopy types.

If an oriented $S^3$-fibration over a 
finite complex $Z$
has a section, which is true for
$S(\hat\eta)$, $S(\hat\z)$, then  
it is classified by a map $[Z, BSG_4]$
that factor through $BSF_3$~\cite[Proposition 9.20(i)]{Ran-book}. 
Since $SF_3$ is a component of $\Omega^3S^3$, the space
$BSF_3$ is rationally contractible, so $[Z, BSF_3]$ is finite.
Thus for all choices of parameters $a, b, q, r$ the $S^3$-fibrations
$S(\hat\eta)$, $S(\hat\z)$ lie in finitely many fiber homotopy types; 
in particular, $M_{\g,\eta}$ lie in finitely many homotopy types. 
\end{proof}

\begin{rmk}
It is instructive 
to see why the argument at the end of the proof 
fails for oriented $S^2$-fibrations with a section: 
the classifying map in $[Z, BSG_3]$ only factors through
$BSF_2$ and the inclusion $BSF_2\to BSG_3$ is rationally 
equivalent to $BSO_2\to BSO_3$~\cite{Han}
while $[Z, BSO_2]\to [Z,BSO_3]$
has infinite image for $Z=S^2\times S^2$ corresponding to
classifying maps in $[Z, BSO_2]$ of
circle bundle with nonzero $e$ and $p_1$. This is the reason
we have to assume $\z$ has rank $\ge 4$. Similarly, in
Theorem~\ref{thm: codim 4 souls over S4}  
we assume $\xi_m^n$ has rank $n\ge 4$ because
$S^2$-bundles over $S^4$ with structure group $SO_3$ lie in infinitely
many fiber homotopy classes; indeed
the inclusion $BSO_3\to BSG_3$ is a rational
isomorphism~\cite{Han}, and $\pi_4(BSO_3)=\mathbb Z$.
\end{rmk}

\begin{rmk} 
\label{rmk: kpt with nonzero Euler class}
In view of Remark~\ref{rmk: normal euler cl} 
one wants to have a version
of Theorem~\ref{thm: kpt modified} for which the normal
Euler class to the soul is nontrivial. As in~\cite{KPT} this is achieved
by modifying the above proof to work for
$\z$ equal to the Whitney sum of the line bundles over $X$ with 
Euler classes $(0,-q, r)$ and $(0,c,c)$ where $c, q, r$ are nonzero
integers, $c$ is fixed, and $r=q+1$. Indeed,  
\[
e(\hat\z)=
(-q{\bf y}+r{\bf z})(c{\bf y}+c{\bf z})=
c(r-q){\bf y}{\bf z}=c\,{\bf y}{\bf z},
\] so since the Euler class 
determines an oriented spherical fibration up to finite ambiguity, 
there are finitely many fiber homotopy possibilities for $S(\hat\z)$. 
Now 
\[
p_1(\hat\z\oplus\hat\eta)=p_1(\hat\z)+p_1(\hat\eta)=
(2c-2qr+2qr){\bf y}{\bf z}=2c\,{\bf y}{\bf z},
\] 
so $\pi_{\hat\eta}^\ast(p_1(\hat\z\oplus\hat\eta))$
is constant, hence $D(\tilde\z)$ lie in finitely many 
diffeomorphism types. The rest of the proof is the same.
Finally, note that the normal bundle to the soul
has nonzero Euler class: 
$e(\tilde\z)=\pi_{\hat\eta}^\ast(c\,{\bf y}{\bf z})=ca\,{\bf w}{\bf u}$.  
\end{rmk}

\begin{rmk}
More examples of manifolds with
infinitely many souls can be obtained from 
Theorems~\ref{thm: codim 4 souls over S4}--\ref{thm: kpt modified} 
by taking products with suitable complete nonnegatively 
curved manifold $L$.
The only point we have to check is that the souls in the product
are pairwise non-homeomorphic, which is true e.g. if the soul of $L$
has trivial first Pontryagin class; then the souls in the product
are not homeomorphic because their $p_1$'s are different 
integers multiples of primitive elements, and this property
is preserved under any isomorphism of their
$4$th cohomology groups.
\end{rmk}

\begin{prob}
Find a manifold $N$ with a infinite sequence of complete metrics 
$g_k$ of $\sec\ge 0$ satisfying one of the following:\newline
\textup{(i)} souls of $(N, g_k)$ are pairwise non-diffeomorphic 
and have codimension $\le 3$;\newline
\textup{(ii)} souls of $(N, g_k)$ are all diffeomorphic 
while the pairs $(N,\,\textup{soul of}\ g_k)$ are pairwise non-diffeomorphic.
\end{prob}

Examples as in (ii) only without nonnegatively curved metrics
can be found in~\cite[Appendix A]{BK-Mathann}. 

\begin{prob}
Find a manifold $N$ with two complete metrics 
of $\sec\ge 0$ whose souls $S$, $S^\prime$
are diffeomorphic and have codimension $\le 3$, while
the pairs $(N,S)$, $(N, S^\prime)$ are not diffeomorphic.
\end{prob}

\section{Vector bundles with diffeomorphic total spaces}
\label{sec: vb with diffeo tot spaces}

One of the things we are unable to do in this paper is 
construct a manifold that admits metrics with 
infinitely many nondiffeomorphic souls of codimension $\le 3$. To get an idea
what such a manifold could look like, in this section we systematically
study vector bundles with diffeomorphic total spaces, especially those
of rank $\le 3$.

Throughout this section $N$ is the total space of
vector bundles $\xi$, $\eta$ over closed 
manifolds $B_\xi$, $B_\eta$, respectively. 
Composing the zero section of $\xi$ with
the projection of $\eta$ gives
a canonical homotopy equivalence $f_{\xi,\eta}\co B_\xi\to B_\eta$.

The map $f_{\xi,\eta}$ 
pulls $TN\vert_{B_\eta}$ to $TN\vert_{B_\xi}$ 
because the projection of $N\to B_\eta\hookrightarrow N$ 
is homotopic to ${\bf id}(N)$.

Any homotopy equivalence  of 
closed manifold preserve Stiefel-Whitney classes, 
as follows from their definition 
via Steenrod squares, so
$f_{\xi,\eta\,}^\ast w(TB_\eta)= w(TB_\xi)$. Therefore,
the Whitney sum formula implies that $f_{\xi,\eta}$ 
also pulls back the normal total Stiefel-Whitney class $w$, i.e.
$f_{\xi,\eta\,}^\ast w(\eta)\cong w(\xi)$.
In fact, Stiefel-Whitney classes of a vector bundle
depend on the fiber homotopy
type of its sphere bundle. To this end we show:

\begin{prop} \label{prop: normal bundles are fiber homotopic}
There is a fiber homotopy equivalence 
$S(f_{\xi,\eta}^\#\eta)\cong S(\xi)$.
\end{prop}
 
It follows that $f_{\xi,\eta}$ pulls back the  
normal Euler classes
(with any local coefficients).

\begin{proof}[Proof of Proposition~\ref{prop: normal bundles are fiber homotopic}]
Use some metric on the fibers to
choose tubular neighborhoods $D_r(\eta)$, $D_\r(\xi)$, $D_R(\eta)$ 
of the zero sections of $\eta$, $\xi$, $\eta$, respectively such that 
$D_r(\eta)\Subset D_\r(\xi)\Subset D_R(\eta)$.
In the commutative diagram below unlabeled arrows are
either inclusions or sphere/disk bundle projections, $p$ is the
obvious retraction along radial lines, and $p(S_\r(\xi))\subset S_r(\eta)$
because of the above inclusions of disk bundles. 
\[
\xymatrix{
B_\xi\ar[r] & D_\r(\xi)\ar[r]&
D_R(\eta)\ar[r]_{p} & D_r(\eta)\ar[r]& B_\eta\\
& S_\r(\xi)\ar[u]\ar[ul]\ar[rr]^{p\vert_{S(\xi)}}&& S_r(\eta)\ar[u]\ar[ur]
}
\]
The composition of top arrows is $f_{\xi,\eta}$, which
by commutativity is covered by $p\vert_{S_\r(\xi)}$. 
By a criterion in~\cite[Theorem 6.1]{Dol} to show that $p\vert_{S_\r(\xi)}$
induces a fiber homotopy equivalence of $S_\r(\xi)$ and
the pullback of $S_r(\eta)$ via $f_{\xi,\eta}$, it is enough to
check that $p\vert_{S_\r(\xi)}$ is a
homotopy equivalence.
Lemma~\ref{lem: s1-bundles are hcobordant} below implies that
$W_R:=D_R(\eta)\setminus{\mathring D}_\r(\xi)$ 
and $W_r:=D_\r(\xi)\setminus{\mathring D}_r(\eta)$ 
are h-cobordisms with ends $S_R(\eta)$, $S_\r(\xi)$ and
$S_\r(\xi)$, $S_r(\eta)$, respectively. 
Therefore, the inclusion of $S_\r(\xi)$ 
into the trivial h-cobordism
$W:=W_R\cup W_r=D_R(\eta)\setminus{\mathring D}_r(\eta)$
is a homotopy equivalence, and so is $p\vert_{W}\co W\to S_r(\eta)$,
hence $p\vert_{W_r}$ defines a deformation retraction
$D_\r(\xi)\to D_r(\eta)$ that restricts to the homotopy
equivalence $p\vert_{S_\r(\xi)}\co S_\r(\xi)\to S_r(\eta)$.
\end{proof}

\begin{cor}\label{cor: codim 2 and tang hom eq}
If $\xi$ has rank $i\in\{1,2\}$, then $f_{\xi,\eta}^\# \eta\cong\xi$,
and $f_{\xi,\eta}$ is tangential.
\end{cor}
\begin{proof}
Since $O_i\to G_i$ is a homotopy equivalence, 
the fiber homotopy equivalence of $f_{\xi,\eta}^\# S(\eta)$
and $S(\xi)$ is induced by an isomorphism of 
$f_{\xi,\eta}^\# \eta\cong\xi$. Thus
$\xi\oplus TB_\xi=TN\vert_{B_\xi}=f_{\xi,\eta}^\# TN\vert_{B_\eta}=
f_{\xi,\eta}^\# (\eta\oplus TB_\eta)\cong
\xi\oplus f_{\xi,\eta}^\# TB_\eta$. Subtracting $\xi$
we see that $f_{\xi,\eta}$ 
pulls back stable tangent bundles.
\end{proof}

In codimension $3$ all we can say is that $f_{\xi,\eta}$ pulls
back rational Pontryagin classes of normal and tangent bundles;
recall that a stable vector bundle is determined by its rational
Pontryagin classes up to finite ambiguity. 

\begin{prop} \label{prop: codim 3 and pontr class}
If $\xi$ has rank $3$, and $p$ denotes the rational total Pontryagin class, 
then  $f_{\xi,\eta\,}^\ast p(\eta)\cong p(\xi)$ and
$f_{\xi,\eta\,}^\ast p(TB_\eta)\cong p(TB_\xi)$.
\end{prop}

\begin{proof}
By Proposition~\ref{prop: normal bundles are fiber homotopic}, 
and Lemma~\ref{lem: p_1 is fiber homotopy inv} below, 
$f_{\xi,\eta\,}^\ast p_1(\eta)\cong p_1(\xi)$, 
while the higher Pontryagin classes vanish as 
$H^\ast(BSO_3;\mathbb Q)\cong{\mathbb Q}[p_1]$.
Now
$f_{\xi,\eta}^\# TN\vert_{B_\eta}\cong TN\vert_{B_\xi}$
and the Whitney sum formula gives
$f_{\xi,\eta\,}^\ast p(TB_\eta)\cong p(TB_\xi)$.
\end{proof}

\begin{prop}\label{prop: h.e. on bases has trivial norm inv}
If $\xi$ has rank $2$, then $f_{\xi,\eta}^\# \eta\cong\xi$ and
$f_{\xi,\eta}$ has trivial normal invariant.
\end{prop}
\begin{proof}
To avoid clutter we set $f:=f_{\xi,\eta}$ and
label associated bundles with a superscript,
e.g., unit disk and sphere bundles for $\xi$ are denoted by $D_\xi$, $S_\xi=\d D_\xi$.
Also let $\e$ denote the pullback of $\eta$ via $f$.

Use metrics on $\xi$, $\eta$ to find their disk bundles
that satisfy $D_\xi\Supset D_\eta\Supset B_\xi$. 
Lemma~\ref{lem: s1-bundles are hcobordant} below implies that
$D_\xi\setminus{\mathring D}_\eta$ is an h-cobordism, so
there exists a deformation retraction 
$r\co D_\xi\to D_\eta$. 
By the proof of Proposition~\ref{prop: normal bundles are fiber homotopic}
the restriction of $r$ to the boundary is a fiber homotopy
equivalence of the sphere bundles $S_\xi\to S_\eta$
covering $f$.
The corresponding homotopy equivalence of
pairs  $r\co (D_\xi, S_\xi)\to (D_\eta, S_\eta)$ has trivial normal invariant
because $D_\xi\times I$ can be thought of as an h-cobordism
with boundaries $D_\xi$, $D_\eta$
(cf.~\cite{Wal-book} before theorem 1.3), and moreover,
the map $D_\xi\times I\to D_\eta$ given by composing
the coordinate projection with $r$ defines a normal bordism
of $r$ and ${\bf id}(D_\eta)$.

Let $\hat f\co (D_{\e}, S_{\e})\to (D_\eta, S_\eta)$ be the homotopy equivalence of pairs
induced by pullback via $f$. 
If $\bar r$ is a homotopy inverse of $r$, then 
$\bar r\circ \hat f\vert_{S_{\e}}\co S_{\e}\to S_\xi$ is a
a fiber homotopy equivalence covering a map homotopic to
the identity of $B_\xi$. Lifting the homotopy to $S_\xi$ and inserting it 
to a collar neighborhood of $\d D_\xi$ we can assume that 
$\bar r\circ \hat f\vert_{S_{\e}}\co S_{\e}\to S_\xi$ covers
the identity map of $B_\xi$.

Any fiber homotopy equivalence of circle fibrations is induced by a linear isomorphism
because $G_2/O_2$ is contractible
(as follows e.g., by applying the $5$-lemma to the inclusion induced map
of the homotopy exact sequences 
of the fibrations $O_2\to S^1$ and $G_2\to S^1$,
and using that the inclusion  $O_1\to F_1$
of their fibers is a homotopy equivalence, cf.~\cite[Corollary 3.5]{MadMil}.)

Since $S_\xi$, $S_\e$ are circle bundles, the fiber homotopy equivalence
$\bar r\circ \hat f\vert_{S_{\e}}$ is induced by a linear isomorphism $\e\to\xi$
covering the identity on $B_\xi$.
The inverse of this isomorphism restricted to disk bundles gives a diffeomorphism $h\co D_\xi\to D_{\e}$
such that $\bar r\circ \hat f\circ h$ can be deformed through boundary-preserving maps
to a homotopy self-equivalence  $q$ of the pair $(D_\xi, S_\xi)$
that is the identity on the union of $S_\xi$ and the 
zero section of $\xi$. The map
$q$ is homotopic to the identity through boundary-preserving maps, where
the homotopy can be defined as follows: let $q_0(x)=x$, and for $t\in (0,1]$
let $q_t(x)=tq(x/t)$ if $|x|\le t$ and $q_t(x)=x$ if $|x|\ge t$. 
Hence $\hat f$ is homotopic (as a map of pairs) to $r\circ h^{-1}$. Therefore, $\hat f$ 
has trivial normal invariant, and hence so does $f$ by Lemma~\ref{lem: norm inv pulls back} below. 
\end{proof}

\begin{rmk}
\label{rmk: connect sum exotic sphere}
Surgery theory implies (see e.g.~\cite[Theorem 13.2]{Ran-book}) that 
if $f\co N\to M$ is
a homotopy equivalence of closed smooth simply-connected manifolds
of dimension $n\ge 5$, then $f$ has trivial normal invariant if and only if
$N$ is diffeomorphic to the connected sum of $M$ and a homotopy sphere
$\Sigma^n$ and $f$ is homotopic to the homeomorphism 
$N\cong M\#\Sigma^n\to M\# S^n\cong M$ where the middle map
is the connected sum of ${\bf id}(M)$ with
a homeomorphism $\Sigma^n\to S^n$. 
Thus Proposition~\ref{prop: h.e. on bases has trivial norm inv} implies 
Theorem~\ref{thm: simply-conn and codim 2}.
\end{rmk}

{\bf Remark.}
Triviality of the normal invariant of $f_{\xi,\eta}$ implies that
$f_{\xi,\eta}$ is tangential, or equivalently that $f_{\xi,\eta}^\# \eta$ and $\xi$ 
are stably isomorphic. On the other hand, the existence of an unstable
isomorphism $f_{\xi,\eta}^\# \eta\cong\xi$ does not imply triviality 
of the normal invariant of $f_{\xi,\eta}$. 
An example is given by any pair of
closed smooth simply-connected manifolds $B_1$, $B_2$ of dimension $m\ge 5$
that are tangentially homotopy equivalent and non-homeomorphic. Here
$B_1\times \R^k$, $B_2\times \R^k$ are diffeomorphic provided $2k\ge m+3$, 
see~\cite{Hae61} and~\cite[Theorem 2.2]{Sie-collar}. Any 
homotopy equivalence of $B_1$ and $B_2$ 
clearly preserves the (unstable) trivial normal bundles, and
if the homotopy equivalence had trivial normal invariant, 
then $B_1$, $B_2$ would be homeomorphic,  as in
Remark~\ref{rmk: connect sum exotic sphere}.

\begin{rmk} 
If a homotopy equivalence of closed manifolds $f\co N\to M$ 
has trivial normal invariant, and if $\a$ is a vector bundle over $M$, 
then by Lemma~\ref{lem: norm inv pulls back}
the induced map ${\hat f}\co D(f^\#\a)\to D(\a)$ of disk bundles
has trivial normal invariant. So
by Wall's $\pi-\pi$-theorem $\hat f$ is homotopic to a diffeomorphism
provided $\dim(D(\a))\ge 6$ and the inclusion $S(\a)\to D(\a)$
is a $\pi_1$-isomorphism. The latter holds if the bundle 
$\a$ has rank $\ge 3$. If the rank of $\a$ is $2$, then 
things are a bit more complicated, and we have partial answers
when $M$ is simply-connected of dimension $\ge 5$. 
Namely, if $\a$ is trivial,
and $\dim(M)\ge 5$, then $N\times\mathbb R^2$ is
diffeomorphic to $M\times\mathbb R^2$ if an only if 
$N$ is diffeomorphic to $M$ (see Remark~\ref{rmk: ko=wh=0}). 
If $\a$ is nontrivial, then $\pi_1(S(\a))$ is a finite cyclic group
$\mathbb Z_d$, and 
a surgery-theoretic argument in~\cite[Section 14]{BKS-mod2} shows
that $\hat f$  is homotopic to a diffeomorphism 
except possibly when $d$ is even and $\dim(M)\equiv 1\, \text{mod}\ 4$.
\end{rmk}

The lemmas below are surely known, yet they do not seem to 
be recorded in the literature in the precise form we need.

\begin{lem} \label{lem: p_1 is fiber homotopy inv}
For $SO_3$-bundles over finite complexes,
the first rational Pontryagin class $p_1$ depends only on the
fiber homotopy type of the associated $2$-sphere bundles.
\end{lem}
\begin{proof}
Denote the natural inclusions $O_3\subset G_3$ and 
$SO_3\subset SG_3$ by $j$ and $j_1$ respectively.
The fiber homotopy invariance of $p_1$ will follow,
once we show that $p_1$  lies
in the image of $Bj^*\co H^4(BG_3;\mathbb Q)\to H^4(BO_3;\mathbb Q)$.
Look at the map induced on rational cohomology
by the commutative diagram, whose rows are projections
of the $2$-fold coverings
corresponding to the first Stiefel-Whitney class.
\[
\xymatrix 
{
H^*(BSO_3;\mathbb Q) & 
H^*(BO_3;\mathbb Q) \ar[l]\\
H^*(BSG_3;\mathbb Q)\ar[u]_{Bj_1^\ast}&
H^*(BG_3;\mathbb Q)\ar[l]\ar[u]^{Bj^\ast}&
}
\]
The horizontal arrows are induced by covering projections, 
hence by a standard argument they are monomorphisms onto 
the subspace fixed by the 
covering involution. Now $Bj_1^\ast$ is an isomorphism~\cite{Han}
which is $\mathbb Z_2$-equivariant because $Bj_1$
is $\mathbb Z_2$-equivariant. 
So $Bj^\ast$ is an isomorphism as well.  
\end{proof}

\begin{lem} \label{lem: s1-bundles are hcobordant}
Suppose a manifold $N$ is the total space of two
vector bundles over closed manifolds $M_1$, $M_2$.
If the normal sphere bundles to $M_1$, $M_2$
are chosen to be disjoint in $N$, then the region between 
these sphere bundles is an h-cobordism. 
\end{lem}
\begin{proof} 
Denote by $S_k(r)$, $D_k(r)$ the normal $r$-sphere, $r$-disk bundles 
determined by some metric on the normal bundle to $M_k$; 
denote by $p_k$ the line bundle projection $N\setminus M_k\to S_k(r)$.
Since $D_k(r)$ exhaust $N$, there are positive numbers $r<t<R<T$ 
such that $D_1(r)\Subset D_2(t)\Subset D_1(R)\Subset D_2(T)$.
We are to show that $W:=D_1(R)\setminus \mathring D_2(t)$
is an $h$-cobordism.

To see that $S_2(t)\hookrightarrow W$ is a homotopy equivalence
it suffices to note that $N\setminus\mathring D_2(t)$ deformation retracts 
both to $W$ and to $S_2(t)$ along the fibers of $p_1$, $p_2$,
respectively.

To show that  $S_1(R)\hookrightarrow W$ is a homotopy equivalence, 
we first observe that $S_1(R)\hookrightarrow W$ 
is $\pi_1$-injective for if a loop
in $S_1(R)$ is null-homotopic in $W$, then it would be
null-homotopic in the larger region 
$D_1(R)\setminus \mathring D_1(r)$ which deformation
retracts to  $S_1(R)$ along the fibers of $p_1$,
so the null-homotopy can be pushed to $S_1(R)$.

To see that $S_1(R)\hookrightarrow W$ 
is $\pi_1$-surjective, start with an arbitrary loop $\alpha$ in $W$,
and since $W$ lies in $D_2(T)\setminus\mathring D_2(t)$
which deformation retracts to $S_2(T)$ along the fibers of $p_2$,
the loop $\alpha$ can be homotoped inside $D_2(T)\setminus\mathring D_2(t)$
to some $\beta$ in $S_2(T)$.
Since $N\setminus\mathring D_1(R)$ deformation retracts 
to $S_1(R)$, the homotopy can be pushed to $W$, where
$\beta$ gets mapped into $S_1(R)$. Thus $\alpha$
is homotopic in $W$ to a loop in $S_1(R)$, as claimed.

Since $S_2(t)\hookrightarrow W$ is a homotopy equivalence,
the pair $(W,S_2(t))$ has trivial cohomology for any
system of local coefficients on $W$, hence by Poincar\'e Duality 
the pair $(W,S_1(R))$ has trivial homology for any
system of local coefficients on $W$. 
So by the non-simply-connected version of
Whitehead's theorem, which is applicable since 
$S_1(R)\hookrightarrow W$ induces a $\pi_1$-isomorphism,
we conclude that $S_1(R)\hookrightarrow W$ 
is a homotopy equivalence.
%
%
\end{proof}

\begin{lem}\label{lem: norm inv pulls back} 
For a homotopy equivalence $f\co N\to M$
of closed smooth manifolds, and a vector bundle 
$\a$ over $M$ with projection $p$, 
let $\hat f\co \hat N\to \hat M$ be the induced map
of disk bundles $\hat N:=D(f^\#\a)$, $\hat M:= D(\a)$. 
Then the normal invariants of $f$ and $\hat f$ satisfy
${\mathfrak q}(\hat f)=p^\ast{\mathfrak q}(f)$.
\end{lem}

\begin{proof}
Let $g$ be a homotopy inverse of $f$, and denote by 
$\hat g\co \hat M\to \hat N$ the corresponding map of disk bundles. 
Fix a large $m$ such that $g$ postcomposed with the inclusion 
$N\to N\times \mathbb R^m$ is homotopic to a smooth embedding 
$e\co M\to N\times \mathbb R^m$, where we may assume that
its image is disjoint from $N\times\{0\}$; let $\nu_e$ denote
the normal bundle of $e$. By~\cite[Theorem 2.2]{Sie-collar} the complement
of a tubular neighborhood of $e(M)$ is an open collar, hence
$N\times \mathbb R^m$ can be identified with the total space 
of $\nu_e$. By the proof of 
Proposition~\ref{prop: normal bundles are fiber homotopic} 
there are disjoint normal sphere bundles
$N\times S^{m-1}$, $S(\nu_e)$ to $N\times\{0\}$, $e(M)$, 
respectively, such that the region between them 
is an h-cobordism, and the radial
projection $N\times \mathbb R^m\setminus\{0\}\to N\times S^{m-1}$ 
restricted to $S(\nu_e)$ is a fiber homotopy equivalence.
Projecting on the $S^{m-1}$-factor
gives a fiber homotopy trivialization $t\co S(\nu_e)\to S^{m-1}$.
As indicated in~\cite{Wal-book} after 
Lemma~10.6, the pair $(\nu_e, t)$ 
represents ${\mathfrak q}(f)$, and moreover, there is
a relative version of the above argument which can be
applied to the embedding of disk bundles 
$\hat e\co \hat M\to \hat N\times \mathbb R^m$
obtained as the pullback of $e$.
Again, the radial projection restricts to a fiber homotopy equivalence
$S(\nu_{_{\hat e}})\to \hat N\times S^{m-1}$ of normal sphere 
bundles, which gives rise to a fiber homotopy trivialization 
$\hat t\co S(\nu_{_{\hat e}})\to S^{m-1}$
such that $(\nu_{_{\hat e}}, \hat t)$ 
represents ${\mathfrak q}(\hat f)$.

That $p^\#\nu_e$ is stably isomorphic to $\nu_{_{\hat e}}$
follows by a straightforward computation showing that
$g^\#\nu_{_N}\oplus \tau_{_M}\oplus\nu_e$ and  
${\hat g}^\#\nu_{_{\hat N}}\oplus 
\tau_{_{\hat M}}\oplus \nu_{_{\hat e}}$
are stably trivial, and that $p$ pulls $g^\#\nu_{_N}\oplus \tau_{_M}$
back to ${\hat g}^\#\nu_{_{\hat N}}\oplus \tau_{_{\hat M}}$,
where $\nu_{_X}$, $\tau_{_X}$ denote the stable normal 
and tangent bundles of $X$.
That $\hat t$ is homotopic to $t\tinycirc p$ follows 
because by construction $t=\hat t\vert_{S(\nu_e)}$ 
and $p$ is a deformation retraction.
\end{proof}

\section{Finiteness results for vector bundles of rank $\le 3$}
\label{sec: bundles of rank <4}

In this section we prove Theorem~\ref{intro-thm: codim <4 finiteness}
and other related results. Let $\{B_\eta\}$ be the 
set of closed manifolds such that $N$ is the total 
space of a vector bundle $\eta$ of rank $\le 3$
over some $B_\eta$. All $B_\eta$'s are homotopy 
equivalent to $N$ and hence to each other, 
so $n:=\dim(B_\eta)$ is constant, thus all 
$\eta$'s have rank equal to $\dim(N)-n$.

\begin{prop}
\label{prop: tang and trivial norm inv}
$\{B_\eta\}$ can be partitioned into
finitely many subsets such that if $B_\eta$, $B_\xi$
lie in the same subset, then there is a 
tangential homotopy equivalence $g_{\xi,\eta}\co B_\xi\to B_\eta$ 
such that $g_{\xi,\eta\,}^\ast \eta\cong\xi$, and furthermore,
$g_{\xi,\eta\,}$ has trivial normal invariant in $[B_\eta, F/O]$.
\end{prop}
\begin{proof}
Fix one base manifold $B_{\eta_{_0}}\in \{B_\eta\}$ with normal
bundle $\eta_{_0}$ of rank $3$,
and pullback each $\eta$ to $B_{\eta_{_0}}$ via $f_{\eta_{_0},\eta}$. 
Since $f_{\xi,\eta\,}^\ast p(\eta)\cong p(\eta)$,
and since an $O(k)$-bundle is determined by its rational
Pontryagin classes up to finite ambiguity, there are finitely many possibilities for $f_{\eta_{_0},\eta\,}^\#\eta$. 
Each possibility corresponds to a subset in the partition for
if $f_{\eta_{_0},\eta\,}^\#\eta=f_{\eta_{_0},\xi}^\#\xi$, 
then
$h_{\xi,\eta}:={f_{\eta_{_0},\xi\,}^{-1}}{\tinycirc} f_{\eta_{_0},\eta\,}$
pulls back $\eta$ to $\xi$, 
and also is tangential as it preserves $TN$
restricted to the zero sections.
If the rank is $\le 2$, then by 
Corollary~\ref{cor: codim 2 and tang hom eq},
the partition consists of one subset $\{B_\eta\}$;
for the sake of uniformity we set 
$h_{\xi,\eta}:=f_{\xi,\eta}$.

Now fix an arbitrary $B_\eta$ in a subset $S$ of the partition. 
The homotopy $h_{\xi,\eta}$ equivalences represent elements 
$(h_{\xi,\eta}, B_\xi)$ in the structure set 
${\bf S}^h(B_\eta)$ of homotopy structures.
Since $h_{\xi,\eta}$ are tangential, their normal
invariants are in the image of
$[B_\eta, F]\to [B_\eta, F/O]$, which is finite,
so by partitioning $S$ further (into finitely many subsets)
we get a subset $\{B_{\xi_i}\}$ such that 
all $f_i:=h_{\xi_i,\eta}$ have equal normal invariants.
It is for this subset we find the homotopy equivalencies with
desired properties. Fix any $\xi_0\in\{\xi_i\}$.
If $g_{0}$ denotes the homotopy inverse of $f_0$, then
by the composition formula for normal invariants,
recalled in~\cite[Sections 3]{BKS-mod2}, we get 
${\mathfrak q}({g_0}{\tinycirc} f_i)=
f_0^\ast {\mathfrak q}(f_i)+ {\mathfrak q}(g_0)$, and
$0={\mathfrak q}({g_0}{\tinycirc} f_0)=
f_0^\ast {\mathfrak q}(f_0)+{\mathfrak q}(g_0)$.
So ${\mathfrak q}({g_0}{\tinycirc} f_i)=f_0^\ast ({\mathfrak q}(f_i)-{\mathfrak q}(f_0))=0$ as ${\mathfrak q}(f_i)={\mathfrak q}(f_0)$
for all $i$. Now 
$g_{\xi,\eta\,}:={g_0}{\tinycirc} f_i=
{h_{\xi_0,\eta}^{-1}}{\tinycirc}{h_{\xi_i,\eta}}$
is a tangential homotopy equivalence from $B_{\xi_i}$ to $B_{\xi_0}$
that pulls back $\xi_i$ to $\xi_0$, and has trivial normal invariant.
\end{proof}

\begin{thm}\label{thm: 1-connected or even dim pair finiteness}
If $n=\dim(B_\eta)\ge 5$, then the pairs 
$(N,B_\eta)$ lie in finitely many diffeomorphism types if either
\textup{(i)} $B_\e$ is simply-connected, or \newline
\textup{(ii)} $B_\eta$ is orientable, $n$ is even, 
and $G:=\pi_1(B_\eta)$ is finite.
\end{thm} 
\begin{proof} 
Fix $B_\eta$ in a subset of the partition given by 
Proposition~\ref{prop: tang and trivial norm inv}.

(i) Since $B_\eta$ is simply-connected, 
each homotopy equivalence $g_{\xi,\eta\,}$
is simple. By Proposition~\ref{prop: tang and trivial norm inv}
and exactness of the surgery sequence,
the element represented by $g_{\xi,\eta\,}$ in
the simple structure set ${\bf S}^s(B_\eta)$ lie
in the orbit of $L^s_{n+1}(1)$ of the identity, which is finite 
because $L^s_{n+1}(1)$-action factors through the finite group
$bP_{n+1}$.
If $\{B_\eta\}$ is infinite, which is the only interesting case,
then there is a subsequence in which all $g_{\xi,\eta\,}$
represent the same element in the structure set. So there are
diffeomorphisms $\phi_{\xi,\zeta}$ such that
$g_{\zeta,\eta\,}{\tinycirc}\phi_{\xi,\zeta}$ is homotopic to
$g_{\xi,\eta\,}$. But diffeomorphisms pull back normal bundles,
so $\phi_{\xi,\zeta}$ induces a diffeomorphism 
$(N,B_\xi)\to (N,B_\eta)$. 

(ii) 
The homotopy equivalence $g_{\xi,\eta\,}$
represent the elements in the structure set ${\bf S}^h(B_\eta)$ 
that, because of exactness of the surgery sequence and 
Proposition~\ref{prop: tang and trivial norm inv}, lie
in the orbit of $L^h_{n+1}(G)$ of the identity. 
Since $n$ is even and $G$ is finite,
$L^h_{n+1}(G)$ is a  finite group (see e.g.~\cite{HamTay}). 
If $\{B_\eta\}$ is infinite, then 
there is a subsequence in which all $g_{\xi,\eta\,}$
represent the same element
${\bf S}^h(B_\eta)$, and hence they are pairwise h-cobordant.
In particular, there are homotopy equivalences
$\phi_{\xi,\zeta}$ that identify the boundaries $B_\xi$ and $B_\zeta$
of the h-cobordism $W_{\xi,\zeta}$ such that
$g_{\zeta,\eta\,}{\tinycirc}\phi_{\xi,\zeta}$ is homotopic to
$g_{\xi,\eta\,}$. 
Fix orientations on all $B_\eta$ that are preserved by
$f_{\zeta,\eta\,}$; then
by a well-known computation (recalled in Lemma~\ref{lem: h-cob}), 
the torsion
$\tau(\phi_{\xi,\zeta})$ is of the form 
$(-1)^n\s^\ast-\s\in \mathrm{Wh}(G)$ where 
$\s=\tau(W_{\xi,\zeta}, B_\zeta)$, the torsion of 
the pair $(W_{\xi,\zeta}, B_\zeta)$.
By a result of Wall~\cite[7.4, 7.5]{Oli},
finiteness of $G$ implies that the standard involution $\ast$
acts trivially on the quotient of $\mathrm{Wh}(G)$ 
by its maximal torsion subgroup 
${\mathrm SK}_1(\mathbb ZG)$, which is finite, so 
$\s^\ast-\s$ lies in ${\mathrm SK}_1(\mathbb ZG)$,
and hence passing to a subsequence we may assume that
$\tau(\phi_{\xi,\zeta})$ is constant. 

Fix $\zeta$ and vary $\xi$, i.e. let $\xi=\xi_i$.
By the composition formula for torsion, we see that 
${(\phi_{\xi_0, \z})^{-1}}{\tinycirc}{\phi_{\xi_i, \z}}
\co B_{\xi_i}\to B_{\xi_0}$
has trivial torsion, which by the above-mentioned computation
equals the torsion of the h-cobordism $W_{\xi,\xi_0}$
obtained by concatenating the corresponding h-cobordisms
$W_{\xi,\zeta}\cup W_{\xi_0,\zeta}$; thus
$W_{\xi,\xi_0}$ is trivial by the s-cobordism theorem.
It follows that 
${(\phi_{\xi_0, \z})^{-1}}{\tinycirc}{\phi_{\xi_i, \z}}$
is homotopic to a diffeomorphism, which then pull back normal bundles,
and hence induces a diffeomorphism 
$(N,B_{\xi_i})\to (N,B_{\xi_0})$.
\end{proof}

\begin{ex}\label{ex: finiteness fails}
Part (ii) fails for $n$ odd (even though it is unclear how to
realize any
of the following examples in the nonnegative
curvature setting): 

(1) 
If $G$ is finite cyclic of order $5$ or 
of order $\ge 7$, and $M$ is a homotopy lens space
with fundamental group $G$ and dimension $\ge 5$, 
then the h-cobordism class
of $M$ contains infinitely many non-homeomorphic 
manifolds~\cite[Corollary 12.9]{Mil-wh} distinguished
by Reidemeister torsion. Thus these manifolds are not even simply
homotopy equivalent, while their products with 
$\mathbb R$ are diffeomorphic. Also see~\cite{KS-eq-h-con}
for a similar result for fake spherical space forms.

(2) By a result of L{\'o}pez de Medrano~\cite{LdM} there are
infinitely many homotopy $\mathbb {RP}^{4k-1}$'s with $k>1$ 
distinguished by Browder-Livesay invariants, and such that
their canonical line bundles are diffeomorphic to the
canonical line bundle over the standard $\mathbb {RP}^{4k-1}$.

(3) More generally, Chang-Weinberger showed in~\cite{ChaWei}
that for any compact oriented smooth $(4r-1)$-manifold $M$, 
with $r\ge 2$
and $\pi_1(M)$ not torsion-free, there exist infinitely many 
pairwise non-homeomorphic closed smooth manifolds $M_i$ that
are simple homotopy equivalent and tangentially homotopy 
equivalent to $M$. 
As in the proof of Proposition~\ref{prop: equal norm inv implies diffeo}
below we see that the manifolds $M_i\times D^3$ lie in finitely
many diffeomorphism types.
\end{ex}

\begin{rmk} If the closed 
manifolds in (1)-(3) admit metrics of $\sec\ge 0$,
then they can be realized as souls of codimension $3$
with trivial normal bundle because they lie in finitely
many tangential homotopy types so 
Proposition~\ref{prop: equal norm inv implies diffeo} applies;
in fact, examples in (1)-(2)
could then be realized as codimension $1$ souls
because any real line bundle over a closed nonnegatively curved manifold
admits a complete metric of $\sec\ge 0$ with zero section 
being a soul.
\end{rmk}

\begin{rmk} It is an (obvious) implication of 
Theorem~\ref{thm: 1-connected or even dim pair finiteness}(ii) that
examples (1)-(3) disappear
after multiplying by a suitable closed manifold, i.e.
if $B_\eta$, $B$ are orientable, odd-dimensional, closed manifolds 
with finite fundamental groups, then the pairs
$\{(N\times B, B_\eta\times B)\}$ lie in finitely many
diffeomorphism types. 
\end{rmk}

\begin{rmk} 
It seems the assumption in (ii) that $B_\eta$ is orientable could be 
removed by working with the proper surgery exact sequence but we
choose not to do this here. 
\end{rmk}

The case of trivial normal bundles deserves special attention 
e.g. because the total space always admits a metric of 
$\sec\ge 0$ provided the base does. 

\begin{prop} \label{prop: equal norm inv implies diffeo}
Suppose that $\{M_i\}$ is a sequence of closed manifolds 
of dimension $\ge 5$. Then $\{M_i\}$ lies in finitely many
tangential homotopy types if and only if $\{M_i\times\mathbb R^3\}$
lies in finitely many diffeomorphism types. 
\end{prop}
\begin{proof} 
The ``if'' direction follows from 
Proposition~\ref{prop: tang and trivial norm inv}, so
we focus on the other direction.
As in the proof of 
Proposition~\ref{prop: tang and trivial norm inv},
after passing to a subsequence we may
assume there are homotopy equivalences $h_i\co M_i\to M_0$
with trivial normal invariants.
Then $H_i:=h_i\times {\bf id}(D^3)$ is normally cobordant to the identity
via the product of $D^3$ with the normal cobordism from $h_i$ to the identity.

The homotopy equivalence $H_i$ need not be simple, so we
replace it with a simple homotopy equivalence as follows.
Attaching a suitable h-cobordism on the boundary of 
$M_i\times D^3$ turns $M_i\times D^3$ into a manifold 
$Q_i$ with ${\mathring Q}_i=M_i\times\mathbb R^3$, and 
precomposing $H_i$ with a deformation
retraction $Q_i\to M_i\times D^3$ yields a simple
homotopy structure equivalence $F_i\co Q_i\to N_0\times D^3$.
(Indeed, if $W$ is a cobordism with a boundary component $M$, then
$\tau(W,M)=-\tau(r)$ where $r\co W\to M$ is a deformation
retraction. By the composition formula for torsion~\cite{Coh}
$\tau(f\tinycirc r)=\tau(f)+f_\ast\tau(r)$, so we need to find $W$ 
with $\tau(f)=f_\ast\tau(W,M)$, or equivalently, since $f_\ast$
is an isomorphism we need $(f_\ast)^{-1}\tau(f)=\tau(W,M)$,
which can be arranged as any element in $\mathrm{Wh}(\pi_1(M)$ 
can be realized as $\tau(W, M)$ for some $W$.)
 
Note that $F_i$ still has trivial normal invariant,  
because $Q_i\times I$ can be thought of as an h-cobordism
with boundaries $Q_i$, $M_i\times D^3$, 
so the maps $F_i$, $H_i$ are normally cobordant 
(as explained in~\cite{Wal-book} before theorem 1.3).
By Wall's $\pi-\pi$-theorem $Q_i$ is diffeomorphic to 
$M_0\times D^3$. Restricting to interiors gives a desired 
diffeomorphism
$M_i\times\mathbb R^3\to M_0\times\mathbb R^3$.
\end{proof}

\begin{rmk}\label{rmk: exotic sphere times R3}
The proof shows that if a homotopy equivalence has trivial normal
invariant, then its product with ${\bf id}(D^3)$ is homotopic
to a diffeomorphism, e.g this implies to the standard 
homeomorphism $\Sigma^k\to S^k$ where $\Sigma_k$ is a homotopy
sphere that bounds a parallelizable manifold, so that 
$\Sigma^k\times\ D^3$ and $S^k\times D^3$ are diffeomorphic.
\end{rmk}

\begin{rmk}
Propositions~\ref{prop: tang and trivial norm inv}, 
\ref{prop: equal norm inv implies diffeo}
imply that if there exist a manifold $N$ with 
infinitely many codimension $\le 3$ souls $S_i$, then after 
passing to a subsequence $S_i\times \mathbb R^3$
are all diffeomorphic, so one also gets infinitely many
codimension $3$ souls with trivial normal bundles,
which proves part (2) of Theorem~\ref{intro-thm: codim <4 finiteness}.
By contrast Proposition~\ref{prop: finite ambig} below shows that
an analogous statement fails in codimension $2$:
indeed, for $r>1$ there are infinitely many homotopy
$\mathbb {RP}^{4r-1}$ with diffeomorphic canonical line 
bundles, hence the products of the line bundles with $\mathbb R$
are also diffeomorphic, yet no two non-diffeomorphic 
homotopy $\mathbb {RP}^{4r-1}$ are h-cobordant 
because $\mathrm{Wh}(\mathbb Z_2)=0$.
\end{rmk}

\begin{prop}\label{prop: finite ambig}
Suppose $\{M_i\}$ is a sequence of closed 
manifolds of dimension $\ge 5$ with finite fundamental group $G$. 
Then $\{M_i\}$ lies in
finitely many h-cobordism types if and only if 
$\{M_i\times\mathbb R^2\}$ lies in finitely many diffeomorphism types.
\end{prop}

\begin{proof}
The ``only if'' direction follows because $M_i$, $M_j$
are h-cobordant, then their products with $\mathbb R$ are diffeomorphic
(e.g. by the weak h-cobordism theorem). For the ``if'' direction
we pass to a subsequence so that all $M_i\times\mathbb R^2$ are
diffeomorphic

By Lemma~\ref{lem: s1-bundles are hcobordant}
each $M_i\times S^1$ is $h$-cobordant to $M_0\times S^1$, 
and the proof of 
Proposition~\ref{prop: normal bundles are fiber homotopic}   
shows that the fiber homotopy equivalence 
$M_i\times S^1\to M_0\times S^1$ 
covers the canonical homotopy equivalence $M_i\to M_0$, so
the circle factor is preserved up to homotopy. Passing to the infinite
cyclic cover corresponding to the circle factor, 
we get a proper h-cobordism
between $M_i\times\mathbb R$ and $M_0\times\mathbb R$.
The proper $h$-cobordisms with one end $M_0\times\mathbb R$
are classified by $\tilde K_0(\mathbb ZG)$~\cite{Sie}.
The group $G$ is finite, and hence so is $\tilde K_0(\mathbb ZG)$~\cite{Swa}. 
Thus $M_i\times \mathbb R$ fall into finitely many 
diffeomorphism classes, and hence $\{M_i\}$ is finite 
up to h-cobordism because closed manifolds 
are h-cobordant if their products with $\mathbb R$ are diffeomorphic
(this is well-known and follows from
Lemma~\ref{lem: s1-bundles are hcobordant}). 
\end{proof}

\begin{rmk}
Milnor noted in~\cite[Theorem 11.5]{Mil-wh} 
that if $B$ is a closed orientable
manifold with finite fundamental group and 
even dimension $\ge 5$, then 
the h-cobordism class of $B$ 
contains only finitely many diffeomorphism classes. 
\end{rmk}

\begin{rmk}\label{rmk: ko=wh=0}
It is well-known that 
a closed simply-connected manifold of dimension $\ge 5$
can be recovered from its product with $\mathbb R^2$.
The proof of Proposition~\ref{prop: finite ambig} 
immediately gives a slight generalization: 
if $M_0$, $M_1$ are closed $n$-manifolds with $n\ge 5$ and
$\pi_1(M_0)=G=\pi_1(M_1)$ such that $M_1\times\mathbb R^2$ and 
$M_0\times\mathbb R^2$ are diffeomorphic, and
$\mathrm{Wh}(G)=0=\tilde K_0(\mathbb ZG)$,
then $M_1$ and $M_0$ are diffeomorphic.
This applies e.g. if $G$ is $\mathbb Z_n$ for $n=2, 3, 4, 6$,
$\mathbb Z_2\times \mathbb Z_2$, $D_{2m}$ for $m=3, 4, 6$,
as well as if $G$ is torsion-free and virtually abelian
(see references in~\cite{LucSta}). 
\end{rmk}

\section{Smooth knots and disconnectedness of moduli spaces}
\label{sec: smooth knots}

In this section we use results of Haefliger and Levine on
smooth knots to investigate 
how many components of $\mathfrak R^u_{\sec\ge 0} (N)$
can be visited by the $\mathrm{Diff}(N)$-orbit of a given metric,
and to give an example of disconnected 
$\mathfrak M^u_{\sec\ge 0} (N)$.

We start be recalling some results on smooth knots.
According to Haefliger~\cite{Hae62, Hae66}, 
for $n\ge 5$ and $k\ge 3$,
isotopy classes of (smooth) embeddings of $S^n$ into $S^{n+k}$ 
form an abelian group $\Sigma^{n+k,n}$ under connected sum, which 
vanishes in metastable range (i.e. for $2k>n+3$), 
and equals to $\mathbb Z$ for $n=4r-1$ and 
$k=2r+1$. In general, Levine~\cite{Lev} showed 
that $\Sigma^{n+k,n}$ is either finite, or virtually cyclic,
and the latter occurs if and only if $n=4r-1$ and $3\le k\le 2r+1$.

For $k\ge 3$, Hirsch~\cite[Theorem 8]{Hir}
showed that any smooth embedding of $S^n$ into $S^{n+k}$ 
can be ambiently isotoped to have a
closed tubular neighborhoods equal to the closed tubular neighborhood
$S^n\times D^k$ of a standard $S^n\subset S^{n+k}$. Of course,
different elements of $\Sigma^{n+k,n}$ define
non-isotopic embeddings of $S^n$ into $S^n\times \mathrm{Int}(D^k)$,
because any isotopy of $S^n$ inside $S^n\times \mathrm{Int}(D^k)$ 
is also an isotopy in $S^{n+k}$.

Levine~\cite{Lev}
showed that assigning to each element of 
$\Sigma^{n+k,n}$ the isomorphism class of its normal bundle
defines a homomorphism $\Sigma^{n+k,n}\to \pi_n(BSO_k)$. 
The image $N_0(n,k)$ of this homomorphism is described 
in~\cite[Theorem 6.9]{Lev} as 
the image of the kernel of the stabilization
homomorphism $\pi_n(SG_k, SO_k)\to \pi_n(SG, SO)$
under the boundary map  $\pi_n(SG_k, SO_k)\to \pi_{n-1}(SO_k)$
in the long exact sequence of the pair. 

The group $N_0(n,k)$ is finite for any $k$, $n$. Indeed,
since tubular neighborhoods of the embeddings $e_i\co S^n\to S^{n+k}$  
can be chosen to equal the same $S^n\times D^k$, 
results of Section~\ref{sec: vb with diffeo tot spaces} 
imply that the canonical homotopy equivalence 
$e_i(S^n)\to S^n\times\{0\}$ 
pulls back the normal Euler class, which is trivial.
Also any homotopy equivalence $e_i(S^n)\to S^n\times\{0\}$ 
pulls back the stable normal bundles because they are trivial.
It follows that the normal bundle to $e_i$ has trivial
Euler and Pontryagin classes, which determine an oriented
vector bundle up to finite ambiguity.

\begin{thm} Let $g$ be any complete metric of $\sec\ge 0$ on  
$N:=S^n\times\mathbb R^k$ with soul $S\times\{0\}$. 
If $r\ge 2$ is an integer and
$n=4r-1$ and $3\le k\le 2r+1$, then metrics
that are isometric to $g$ lie in 
infinitely many components
$\mathfrak R^u_{\sec\ge 0} (N)$
and in the same component of $\mathfrak R^u (N)$.
\end{thm}
\begin{proof}

As $\Sigma^{n+k,n}$ is infinite,
and the above-mentioned homomorphism
$\Sigma^{n+k,n}\to \pi_n(BSO_k)$ has finite image,
its kernel contains
infinitely many isotopy classes of embeddings of
$S^n$ into $S^{n+k}$ with trivial normal bundle.
By above-mentioned result in~\cite{Hir}, we ambiently
isotope the embeddings into infinitely many pairwise
non-isotopic embeddings from $S^n$ to $S^n\times \mathrm{Int}(D^k)$,
which is a closed tubular neighborhood for all the 
zero sections. 
Equipping their normal bundles with the metric $g$, we get
infinitely many metrics on $S^n\times\mathbb R^k$
with pairwise non-isotopic souls, and the metrics lie
in different components of 
$\mathfrak R^u_{\sec\ge 0} (S^n\times\mathbb R^k)$ by 
Lemma~\ref{lem: souls amb isot}, and modifying the metrics
as in Proposition~\ref{prop: guijarro modification} 
they can be arranged to lie  
in the same component of $\mathfrak R^u (N)$.
\end{proof}

By convention we treat $S^1$, $\mathbb R$, and a point as 
having $\sec\ge 0$.

\begin{thm}\label{thm: knots mod space}
If $L$ is any closed manifold of $\sec\ge 0$, then the moduli space
$\mathfrak M^u_{\sec\ge 0} (S^7\times \mathbb R^4\times L)$ 
has more than one component with
metrics whose souls are diffeomorphic to $S^7\times L$,
and which lie in the same component of $\mathfrak M^u (N)$.
\end{thm}
\begin{proof}
A key ingredient of the proof is that the group $N_0(7,4)$
is nontrivial. In table 7.2 of~\cite{Lev} Levine stated (without proof)
that $N_0(7,4)$ has order $4$; for completeness 
we justify Levine's assertion using results in~\cite{Hae66} and~\cite{Tod}.  

Since $\pi_7(SG, SO)$ vanishes, $N_0(7,4)$ is
the image of $\pi_7(SG_4, SO_4)\to \pi_{6}(SO_4)$,
or equivalently, the kernel of $\pi_{6}(SO_4)\to \pi_{6}(SG_4)$.
To compute the latter look at the following commutative
diagram in which the rows are exact sequences of the fibrations, and vertical arrows are induced by inclusions.
\[
\xymatrix{
\pi_7(S^3)\ar[r]_{\text{zero}}\ar@{=}[d]&
\pi_6(SO_3)\ar[r]_{1-1}\ar[d]_i&
\pi_6(SO_4)\ar[r]_{\text{splits}}\ar[d]_j&
\pi_6(S^3)\ar@{=}[d]&\\
\pi_7(S^3)\ar[r]_{\text{zero}}&
\pi_6(SF_3)\ar[r]_{1-1}&
\pi_6(SG_4)\ar[r]_{\text{splits}}&
\pi_6(S^3)&
}
\]
The principal $SO_3$-bundle $SO_4\to S^3$ is trivial
(a principal bundle is trivial if it has a section
and obstructions to constructing a section lie in the groups
$H^{s+1}(S^3;\pi_s(SO_3))$ which are trivial. It follows that
the fibration $SG_4\to S^3$ has a section, which explains
how the horizontal arrows are labeled. Furthermore,
$i$ is the sum of $j$ and
the identify of $\pi_6(S^3)$; in particular, $i, j$ have
isomorphic kernels and cokernels, and again,
the kernels are isomorphic to $N_0(n,k)$.
The groups $\pi_{6}(SO_3)$, $\pi_{6}(SF_3)=\pi_9(S^3)$
equal to $\mathbb Z_{12}$, $\mathbb Z_3$, respectively,
so $i$ is either trivial, or onto. In the latter case,
the kernel of $j$ is $\mathbb Z_4$ as desired, so it remains
to show that $i$ cannot be trivial, or equivalently,
that the cokernel of $j$ is not $\mathbb Z_3$.
The cokernel of $j$ lies in $\pi_6(SG_4, SO_4)$, 
which fits in an exact sequence of homotopy groups of the triad 
(see~\cite[4.11]{Hae66})
\[
\pi_7(SG; SO, SG_4)\to \pi_6(SG_4, SO_4)\to\pi_6(SG, SO).
\]
Now $\pi_6(SG, SO)=\mathbb Z_2$ and by~\cite[Theorem 8.15]{Hae66}
$\pi_7(SG; SO, SG_4)=0$, which means that $\pi_6(SG_4, SO_4)$
has no subgroup isomorphic to $\mathbb Z_3$, as promised.

Actually, Haefliger omits the proof that $\pi_7(SG; SO, SG_4)=0$,
so we fill in the details. Consider the exact sequence
given by~\cite[Theorem 6.4]{Hae66}: 
\[
\pi_7(SF_4, SG_4)\to\pi_7(SG; SO, SG_4)\to \pi_7(SG; SO, SG_5)\to 
\pi_6(SF_4, SG_4).
\]
Here $\pi_7(SG; SO, SG_5)=0$ by~\cite[Corollary 6.6]{Hae66}, 
so it remains to see that
$\pi_7(SF_4, SG_4)=0$. To this end consider
the exact sequence of the pair:
\[
\pi_7(G_4)\to\pi_7(F_4)\to\pi_7(F_4,G_4)\to
\pi_6(G_4)\to\pi_6(F_4)\to\pi_6(F_4, G_4)
\]
As mentioned above, the fibration $G_4\to S^3$ has a section
so $\pi_i(G_4)$ splits as 
$\pi_i(F_3)\oplus\pi_i(S^3)=\pi_{i+3}(S^3)\oplus\pi_i(S^3)$.
In particular, $\pi_6(G_4)=\mathbb Z_3\oplus\mathbb Z_{12}$
and also $\pi_6(F_4)=\pi_{10}(S^4)=\mathbb Z_3\oplus\mathbb Z_{24}$.
By~\cite[Theorem 8.11]{Hae66} 
$\pi_6(F_4, G_4)\cong\pi_3(SO, SO_3)$ which equals to 
$\mathbb Z_2$, so exactness at $\pi_6(F_4)$ implies that
$\pi_6(G_4)\to\pi_6(F_4)$ is one-to-one.
On the other hand, the
inclusion $k\co F_3\to F_4$ factors through $G_3$, so it
suffices to show that $k_\ast\co\pi_7(F_3)\to \pi_7(F_4)$ is onto.
As $\pi_7(F_3)=\pi_{10}(S^3)=\mathbb Z_{15}
\cong\pi_{11}(S^4)=\pi_7(F_4)$, it is enough to see
that $k_\ast$ is one-to-one. In fact, the inclusion $F_3\to F$,
which factors through $F_4$, is an isomorphism on $\pi_7$, because
with the above identifications it corresponds to
the iterated suspension homomorphism $\pi_{10}(S^3)\to\pi_7^{\bf S}$,
and the latter homomorphism is an isomorphism at 
primes $3$, $5$ as shown in~\cite[page 177]{Tod}. 

Thus $N_0(7,4)\cong\mathbb Z_4$, and hence there are $4$
different oriented vector bundles over $S^7$ with
total space diffeomorphic to 
$S^7\times\mathbb R^4$. 

A result of Grove-Ziller~\cite{GroZil} implies that
their total spaces admits complete metrics of $\sec\ge 0$
with souls equal to the zero sections, because 
by the discussion preceding Corollary~3.13 of~\cite{GroZil},
all vector bundles in
$\pi_7(BSO_4)\cong\mathbb Z_{12}\oplus\mathbb Z_{12}$
classified by elements of orders $1$, $2$, $4$ 
admit complete metrics with $\sec\ge 0$ and souls equal to the
zero sections, which includes all elements of $N_0(7,4)$.
In particular, there exists a nontrivial bundle $\xi$ with these
properties. 

Thus $S^7\times L\times \mathbb R^4$ is the total space
of the vector bundle $p^\#\xi$, where
$p\co S^7\times L\to S^7$ is the projection on the first factor. 
The bundle $p^\#\xi$ is nontrivial because its pullback
via an inclusion $i\co S^7\to S^7\times L$ is $\xi$, which
is nontrivial.
If some self-diffeomorphism
of $S^7\times L\times \mathbb R^4$ could take 
$S^7\times L\times \{0\}$ to the zero section of 
$p^\#\xi$, then since trivial bundles are preserved by pullback,
it would follow that $p^\#\xi$ is trivial.
Hence by Lemma~\ref{lem: souls amb isot} 
the two metrics lie in different components of the moduli space. 
Modifying the metrics
as in Proposition~\ref{prop: guijarro modification} 
they can be arranged to lie  
in the same component of $\mathfrak R^u (N)$.
\end{proof}

\begin{rmk}
The above proof shows that $N_0(7,4)$ lies in $\pi_6(SO_3)$-factor
of $\pi_6(SO_4)\cong\pi_3(SO_3)\oplus\pi_6(S^3)$,
so we can explicitly write 
$N_0(7,4)=\{0,3,6,9\}\subset\mathbb Z_{12}=\pi_7(BSO_3)$.
Pulling back the bundle represented by 
$m\in\pi_7(BSO_3)$ by an degree $-1$ self-map
of $S^4$ yields the bundle represented by $-m$, so
up to the action of homotopy self-equivalences on the base $S^7$
we get only $3$ different bundles, namely $0,3,6$. It follows that
$\mathfrak M^u_{\sec\ge 0} (S^7\times \mathbb R^4)$ has
at least $3$ components with metrics whose 
souls are diffeomorphic to $S^7$. 
Note that 
since all homotopy $7$-spheres become diffeomorphic after multiplying 
by $\mathbb R^k$ with $k\ge 3$, the moduli 
space $\mathfrak M^u_{\sec\ge 0} (S^7\times \mathbb R^4)$ 
also has other components with souls diffeomorphic to those
homotopy $7$-spheres that admit metrics of $\sec\ge 0$. 
\end{rmk}

\section{Spaces of metrics and h-cobordisms}
\label{sec: h-cob}

In this section we study  
components of $\mathfrak R^u_{\sec\ge 0} (N)$
via various techniques related to h-cobordisms.

Here is a basic idea. If $W$ is an h-cobordism
of dimension $\ge 5$ with boundary components 
$M$, $M^\prime$, then by the weak h-cobordism theorem
there is a diffeomorphism 
from $\mathring W:=\mathrm{Int}(W)$ onto $M\times\mathbb R$,
taking $M$ to $M\times \{0\}$, and similarly for $M^\prime$. 
In particular,
if $M$, $M^\prime$ admit metrics with $\sec\ge 0$,
then $\mathring W$ has two complete metrics of $\sec\ge 0$
with souls isometric to $M$, $M^\prime$. 
Therefore, if $M$, $M^\prime$ are not diffeomorphic,
then $\mathfrak M^u_{\sec\ge 0} (\mathring W)$ is not connected
by Lemma~\ref{lem: souls amb isot}; this implies 
Example~\ref{ex: mil-lens}.
If $M$, $M^\prime$ are diffeomorphic, but 
the h-cobordism $W$ is nontrivial, then
$\mathfrak R^u_{\sec\ge 0} (\mathrm{Int}(W))$ is not connected
because the boundaries
of a nontrivial h-cobordism are not isotopic
(as explained e.g. in~\cite[Lemma 7.3]{BK-acta}).

We use~\cite{Coh, Mil-wh} as basic references on the Whitehead
torsion and h-cobordisms. The following lemma summarizes
the standard formulas that we need.

\begin{lem}
\label{lem: h-cob}
\textup{(i)}
If $W$ is an oriented h-cobordism with boundaries $M$, $M^\prime$,
and $r\co W\to M$ is a deformation retraction, then
the homotopy equivalence
$r\vert_{M^\prime}\co M^\prime\to M$ has torsion 
$-\tau(W, M)+ (-1)^{\dim(M)}\tau^\ast (W, M)$.\newline
\textup{(ii)}
Let $W_1$, $W_2$ be two oriented h-cobordisms attached 
along their common boundary component to form 
an oriented h-cobordism $W$. Denote
the common component of $\d W$ and $\d W_k$ by $M_k$, and let
$i_k\co M_k\to W$ be the inclusion.
Then $i_{1\ast}\tau(W, M_1)=i_{1\ast}\tau(W_1, M_1)+
(-1)^{\dim(M)}i_{2\ast}\tau^\ast (W_2, M_2)$.
\end{lem}
\begin{proof}
(i) 
First recall that if $r\co X\to Y$ is a deformation retraction 
of finite cell complexes, then $\tau(r)=-\tau(X,Y)$, indeed,
if $i\co Y\to X$ is the inclusion, then 
$\tau(r)=-r_\ast\tau(i)=-r_\ast i_\ast\tau(W, M)=-\tau(W, M)$~\cite[22.3, 22.5]{Coh}. Now to prove (i)
let $i$, $i^\prime$ be the inclusions of $M$, $M^\prime$
into $W$ so that~\cite[22.3, 22.4, 22.5]{Coh}
implies that 
\[
\tau({r}{\tinycirc} i^\prime)
=\tau(r)+r_\ast \tau(i^\prime)
=\tau(r)+r_\ast i^\prime_\ast\tau(W, M^\prime)
\]
where $i^\prime_\ast\tau(W, M^\prime)
=(-1)^{\dim(M)} i_\ast \tau^\ast(W, M)$ by 
duality~\cite[page 394]{Mil-wh} where $\ast$
is the standard involution of $\mathrm{Wh}(G)$
induced by $g\to g^{-1}$ in $G$.

(ii) Fix deformation retractions 
$R\co W\to W_1$, $r\co W_1\to M_1$. Below 
we slightly abuse notations by using $i_k$
to also denote the inclusion $M_k\to W_k$.
The proof of (i) and the composition formula
for torsion~\cite[22.4]{Coh} gives
\[
\tau(W, M_1)=-\tau({r}{\tinycirc} R)=-\tau(r)-r_\ast\tau(R)=
\tau(W_1, M_1)+r_\ast\tau(W, W_1)
\]
By excision~\cite[20.3]{Coh} and duality 
$\tau(W, W_1)=(-1)^{\dim(M)} i_{2\ast} \tau^\ast(W_2, M_2)$.
Applying $i_{1\ast}$ to both sides of the equation and recalling
that ${i_1}{\tinycirc} r$ is homotopic to ${\bf id}(W_1)$,
we get the desired formula.
\end{proof}

If $G$ is a finite group, then its Whitehead group
$\mathrm{Wh}(G)$ is fairly
well-understood, in particular,
results of Bass show that $\mathrm{Wh}(G)$ is a finitely
generated abelian group whose rank equals to the difference between
the number of conjugacy classes of subsets $\{g,g^{-1}\}\subset G$ 
and the number of conjugacy classes of cyclic subgroup
of $G$. There is a body of work computing $\mathrm{SK}_1(\mathbb ZG)$,
the (finite) torsion subgroup of $\mathrm{Wh}(G)$, see~\cite{Oli}
for details.

\begin{thm} \label{thm: even dim inf many comp}
Let $M$ be a closed oriented manifold
of even dimension $\ge 5$ with $\sec\ge 0$ such that
$G=\pi_1(M)$ is finite and $\mathrm{Wh}(G)$ is infinite.
Then $\mathfrak R^c_{\sec\ge 0} (M\times\mathbb R)$
has infinitely many components.
\end{thm} 
\begin{proof} Since $\mathrm{Wh}(G)$ is infinite,
$G$ is finite, and $\dim(M)\ge 5$, there is a oriented 
h-cobordism $W_0$ with one boundary diffeomorphic
to $M$ such that $\tau(W_0,M)$ has infinite order in
$\mathrm{Wh}(G)$. We double $W_0$ along the other boundary, 
and denote the double by $W$; thus both boundary components
of $W$ are diffeomorphic to $M$. 
By Lemma~\ref{lem: h-cob} the torsion of the double $W$ is
given by $\tau (W, M)=\tau (W_0, M)+(-1)^{\dim(M)}\tau^\ast (W_0, M)$,
where we suppress inclusions.
By a result of Wall~\cite[7.4, 7.5]{Oli}, the involution
$\ast$ acts trivially on 
$\mathrm{Wh}(G)/\mathrm{SK}_1(\mathbb ZG)$, so since
$\dim(M)$ is even, 
$\tau (W, M)-2\tau (W_0, M)$ has finite order.
Since the order of $\tau (W_0, M)$ is infinite, stacking
$k$ copies of $W$ on top of each other gives a nontrivial h-cobordism
for every $k$.
Stacking countably many copies of $W$ on top of each other,
we get a manifold $W_\infty$ diffeomorphic to $M\times\mathbb R$
and countably many pairwise non-isotopic embeddings
$e_k\co M\to W_\infty$. Since $W_\infty$ is diffeomorphic to a 
tubular neighborhood of every $e_k$, we conclude that for each $k$,
the manifold 
$W_\infty$ carries a metric isometric to $M\times\mathbb R$
with soul $e_k(M)$.
So by Lemma~\ref{lem: souls amb isot} that
the $\mathrm{Diff}(N)$-orbit of the metric $M\times\mathbb R$
visits infinitely many components of
$\mathfrak R^c_{\sec\ge 0} (\mathring W)$.
\end{proof}

\begin{rmk} 
The class of groups with infinite Whitehead group is closed
under products with any group. 
Examples of finite $G$ with 
$Wh(G)$ infinite include $\mathbb Z_m$ with 
$m=5$ or $m\ge 7$, and the dihedral group $D_{2p}$ of order $2p$, 
where $p\ge 5$ is a prime.
See~\cite{Oli} for more information. 
\end{rmk}

For odd-dimensional souls doubling produces h-cobordisms
with finite torsion, so we use a different idea based on
the following addendum to Lemma~\ref{lem: souls amb isot}. 

\begin{prop}\label{prop: def retr and space of metrics}
Let $S$, $S^\prime$ be souls for the metrics  $g, g^\prime$
that lie in the same component of $\mathfrak R^u_{\sec\ge 0} (N)$.
If $P\co N\to S^\prime$ is a deformation retraction, 
then $P\vert_{S}\co S\to S^\prime$ is homotopic to a diffeomorphism.
The same holds for $\mathfrak R^c_{\sec\ge 0} (N)$ if any 
two metrics in the space have souls that intersect.
\end{prop}

\begin{proof} 
The proof below holds for any $k$, so according to  
Convention~\ref{intro-convention} we omit $k$ from notations.
By Lemma~\ref{lem: souls amb isot} any metric 
$g_i\in \mathfrak R_{\sec\ge 0} (N)$ has an open neighborhood
$U_i$ such that for any $g\in U_i$
the Sharafutdinov retraction onto a soul of $g_i$ restricted
to any soul of $g$ is a diffeomorphism.

Fix an arbitrary connected component 
$C$ of $\mathfrak R_{\sec\ge 0} (N)$; thus $\{U_i\}$
is an open cover of $C$.
By a basic property of connected sets~\cite[Section 46, Theorem 8]{Kur} 
for any two $g, g^\prime\in C$ there exists a finite sequence
$g_0=g, g_1,\dots , g_n=g^\prime$ with $g_i\in C\cap U_i$ such that 
$U_i\cap U_{i-1}\neq\emptyset$ for every $0<i\le n$.

Denote souls of $g_i$ by $S_i$
and the corresponding Sharafutdinov retractions by 
$p_i\co N\to S_i$, where $S_0=S$ and $S_n=S^\prime$.
By construction $p_i\co N\to S_i$
restricted to $S_{i-1}$ is a diffeomorphism for every $0<i\le n$.
Since each $p_i\co N\to S_i\subset N$ is 
homotopic to the identity of $N$, the composition  
\[
P_n:=p_n\tinycirc\dots\tinycirc
p_1\co N\to S_n \subset N
\]  
has the same property, and furthermore,
$P_n$ maps $S_0$ diffeomorphically onto $S_n$.
Both $P$ and $P_n$ are homotopic to the identity of $N$,
so if $F$ denotes the homotopy joining $P_n$ and $P$ 
through maps $N\to N$, then $P\tinycirc F$ is a homotopy of 
$P_n$ and $P$ through maps $N\to S_n$. 
Restricting the homotopy to $S_0$, we conclude
that  $P\vert_{S_0}$ is homotopic to a diffeomorphism.
\end{proof}

In the following Proposition we can e.g. take 
$W$ to be Milnor's h-cobordism from Example~\ref{ex: mil-lens}.

\begin{prop}
\label{prop: odd h-cob}
Suppose $W$ is an oriented h-cobordism with boundaries
$M$, $M^\prime$ such that $G:=\pi_1(W)$ is finite,
$\tau(W,M)$ has infinite order in $\mathrm{Wh}(G)$,
and $\dim (M)$ is odd and $\ge 5$. Suppose $L$ is a manifold
with nonzero Euler characteristic. If $M$, $M^\prime$, $L$ admit
complete metrics with $\sec\ge 0$, then 
$\mathfrak R^u_{\sec\ge 0} (L\times\mathring W)$
is not connected.
\end{prop}

\begin{proof}
Let $f\co M^\prime\to M$ be the homotopy equivalence
between the boundary components of $W$ considered in
Lemma~\ref{lem: h-cob}; 
so 
\[
\tau(f)=-\tau(W, M)+ (-1)^{\dim(M)}\tau^\ast (W, M),
\] 
which equals to $-2\tau(W, M)$ plus an element of finite order,
as $G$ is finite and ${\dim(M)}$ is odd.
Set $S$ to be $L$ when $L$ is compact, and 
to be a soul of $L$ if $L$ is non-compact. 
The product formula for torsion~\cite[23.2b]{Coh} implies that
$\tau (f\times{\bf id}(S))$ is mapped to 
$\chi(S)\tau(g)=-2\chi(S)\tau(W, M)$ by
the projection $M\times S\to M$, and  $-2\chi(S)\tau(W, M)$
is nonzero because $\tau(W,M)$ has infinite order and
$\chi(S)=\chi(L)\neq 0$.
Thus $f\times{\bf id}(S)$
is not a simple homotopy equivalence, hence it is not homotopic
to a homeomorphism. Now thinking of $\mathring W$
as the result of attaching to $W$ an open collar along $\d W$, we apply
Proposition~\ref{prop: def retr and space of metrics}
to $L\times\mathring W$.
\end{proof}

\begin{cor} \label{cor: hausmann}
For $k\ge 3$, $r>0$, let
$M=L(4r+1,1)\times S^{2k}$ and let $L$ be a complete manifold
of $\sec\ge 0$ and $\chi(L)\neq 0$. Then
$\mathfrak R^u_{\sec\ge 0} (L\times M\times\mathbb R)$ has infinitely many 
components. 
\end{cor}

\begin{proof} 
Hausmann~\cite{Hau-opn-books} 
showed that there is a nontrivial h-cobordism $W$ with 
both boundaries diffeomorphic to $M=L(4r+1,1)\times S^{2k}$. 
Stacking infinitely many copies
of $W$ on top of each other we get a manifold diffeomorphic
to $M\times\mathbb R$, and infinitely many embeddings 
$M\to M\times\mathbb R$ such that the h-cobordism between
any two distinct embedded copies of $M$ is obtained by stacking
$k$ copies of $W$ on top of each other for some positive integer 
$k$. By Lemma~\ref{lem: h-cob} the homotopy equivalence
between its boundaries has torsion $-2k\tau(W, M)$
and hence it is not homotopic to a diffeomorphism, so that
Proposition~\ref{prop: def retr and space of metrics} applies
yielding the special case when $L$ is a point. The general
case follows as in the proof of Proposition~\ref{prop: odd h-cob}. 
\end{proof}

\section{Acknowledgments}
Belegradek is grateful to A.~Dessai for idea of 
Proposition~\ref{prop: via ricci flow}, V.~Kapovitch for helpful 
conversations about~\cite{KPT}, B.~Wilking and L. Polterovich for useful comments about uniform topology on the moduli space. 
We also appreciate referee's comments on exposition. 
Belegradek was partially supported by the NSF grant \# DMS-0804038,
Kwasik was partially supported by BOR award LEQSF(2008-2011)-RD-A-24.

\small
\bibliographystyle{amsalpha}
\bibliography{mod1-revision}

\end{document}